\documentclass[ECP]{ejpecp} 

\usepackage[utf8]{inputenc}
\usepackage{hyperref} 
\usepackage[mathscr]{eucal}
\usepackage{srcltx} 
\usepackage{dsfont}
\usepackage{enumerate}
\usepackage{lscape}
\usepackage{tabularx}
\usepackage{xcolor}

\usepackage{tommasos-commands}
\usepackage[normalem]{ulem}

\newcommand{\bp}{\boldsymbol{\varphi}}

\SHORTTITLE{Wave speed for jump FKPP}

\TITLE{The wave speed of an FKPP equation with jumps via coordinated branching}

\AUTHORS{%
  Tommaso Rosati\footnote{University of Warwick, UK,
    \EMAIL{t.rosati@warwick.ac.uk}}
  \and 
  András Tóbiás\footnote{Department of Computer Science and Information Theory, Budapest University of Technology and Economics, Műegyetem rkp. 3., 1111 Budapest, Hungary, and Alfréd Rényi Institute of Mathematics, Reáltanoda utca 13-15, 1053 Budapest, Hungary. \EMAIL{tobias@cs.bme.hu} }}


\AMSSUBJ{Fisher-KPP; Coordinated branching; Wave speed; Moment duality.} 

\SUBMITTED{January 20, 2022} 
\ACCEPTED{December 13, 2014} 

\VOLUME{0}
\YEAR{2020}
\PAPERNUM{0}
\DOI{10.1214/YY-TN}

\ABSTRACT{We consider a Fisher--KPP equation with nonlinear selection
driven by a Poisson random measure.
We prove that the equation admits a unique wave speed $ \mf{s}> 0 $ given by
$$ \frac{\mf{s}^{2}}{2} = \int_{[0, 1]}\frac{ \log{(1 + y)}}{y} \mf{R}( \ud y)
\;,$$
where $ \mf{R} $ is the intensity of the impacts of the driving noise. 
Our arguments are based on upper and lower bounds via a quenched duality with a
coordinated system of branching Brownian motions.
}

\numberwithin{equation}{section}
 
\def\emptyset{\varnothing} 

\def\d{{\rm d}} 
\def\e{\varepsilon} 
 
\newfam\Bbbfam 
\font\tenBbb=msbm10 
\font\sevenBbb=msbm7 
\font\fiveBbb=msbm5 
\textfont\Bbbfam=\tenBbb 
\scriptfont\Bbbfam=\sevenBbb 
\scriptscriptfont\Bbbfam=\fiveBbb

\def\2{\mathbf 2}

\newcommand{\R}     {\mathbb{R}} 
 
\newcommand{\N}     {\mathbb{N}} 
\renewcommand{\P}   {\mathbb{P}} 
 
\newcommand{\E}     {\mathbb{E}}

\def\1{{\mathchoice {1\mskip-4mu\mathrm l}      
{1\mskip-4mu\mathrm l} 
{1\mskip-4.5mu\mathrm l} {1\mskip-5mu\mathrm l}}} 
 
\makeatletter
\def\@endtheorem{\endtrivlist}
\makeatother

\newtheoremstyle{thm}{2ex}{2ex}{\itshape\rmfamily}{} 
{\bfseries\rmfamily}{}{1.7ex}{} 
 
\newtheoremstyle{rem}{1.3ex}{1.3ex}{\rmfamily}{} 
{\itshape\rmfamily}{}{1.5ex}{}

\renewcommand{\d}{{\rm d}} 
 
\newcommand{\eps}{\varepsilon}

\newcommand\numberthis{\addtocounter{equation}{1}\tag{\theequation}}
\renewcommand{\e}   {{\operatorname e }}

\begin{document}





\setcounter{section}{0}

\section{Introduction}\label{sec-intro}

The Fisher-KPP equation is a classical model in spatial population genetics 
\begin{equation}\label{eqn:fkpp-deterministic} \partial_t u_t = \frac{1}{2} \Delta u_t + \mathbf r \, u_t(1-u_t), \quad
u(0, x) = u_{0}(x) \in [0, 1], \quad (t, x) \in [0, \infty) \times \RR \;,
\end{equation}
which describes the evolution of the density of one favoured genetic type over
another disadvantaged one, where the advantage is given by a selection force of
strength $ \mathbf{r}> 0 $.

Instead of the classical equation, this work is concerned with the
analysis of a Fisher-KPP model in which selection acts at discrete jump times.
We fix a positive measure $ \mf{R} $ on $ [0, 1] $ and denote with $ \mR$ a
Poisson random measure on $ [0, \infty) \times (0, 1] $ with intensity $
\ud t \otimes \frac{1}{y} \mf{R}( \ud y) $. Then we consider the following equation driven
by $ \mR $:
\begin{equation}\label{eqn:fkpp}
  \ud u_{t} = \frac{1}{2} \Delta u_{t} \ud t + \mathfrak{r} \, u_t (1-u_t) \d t +
\int_{(0,1]} y u_{t-}(1 -
  u_{t-}) \mR (\ud t, \ud y), \quad u(0, x) = u_{0}(x) \;,
\end{equation}
 for $ (t, x) \in [0, \infty) \times \RR$\color{black}, where $ \mf{r} = \mf{R}(\{ 0 \}) $ corresponds to the continuous component of
the equation. The biological motivation behind the choice of such a noise is to model \emph{rare
selection}, or better strong temporary selective advantages of fit individuals due
to extreme behaviour of a random environment, as opposed to the
classical, constantly present but \emph{weak} selection corresponding to models
with continuous forcing (represented e.g.\ by the second term on the right-hand side of \eqref{eqn:fkpp})\color{black}. Such strong evolutionary
events involve a macroscopic portion of the underlying
population. They are therefore linked to some form of ``coordination'' between
individuals. We will elaborate on concrete biological examples of rare selection and mathematical models of coordination in Section~\ref{sec-rareselection}.\color{black}

The purpose of this article is to focus on a prominent
dynamical feature of the Fisher-KPP equation -- its wave speed -- and attempt
a first description of how it is affected by extreme selection events.
For the classical equation it is well known that there exists a travelling wave solution of speed
$\sqrt{2\mathbf r}$, that captures the asymptotic evolution of the front
of an invading gene, say when the initial distribution is of the form {$
u_{0}(x) = \mathds 1_{(- \infty, 0]} (x) $}.  Following the convention that  $\int_{\{0\}} \log(1+y) \frac{1}{y} \mf R(\ud
y) = \mf{R}(\{ 0 \}) = \mf{r}$, our \color{black} main result states that the wave speed $
\mf{s}>0 $ of the stochastic equation is given by
\begin{equation}\label{eqn:intro-wave-speed}
  \frac{\mf{s}^{2}}{2} = \int_{[0,1]} \log{(1+y)} \frac{1}{y} \mf{R}( \ud
  y)\;,
\end{equation}
and therefore shows quantitatively how extreme selection events slow down the invading
speed, compared to the deterministic equation \eqref{eqn:fkpp-deterministic}
with $ \mathbf{r} = \mf{R}([0,1]) $ (this is the natural
choice because it is consistent with our definitions in the case $ \mf{R} =
\mf{r} \delta_{0}$): that the wave speed is strictly smaller than the
deterministic one follows for instance by averaging \eqref{eqn:fkpp} and using
Jensen's inequality.  

The size of the gap
between the speed $ \mf{s} $ of the stochastic equation and the speed $
\sqrt{2 \mathbf{r}} $ of the associated deterministic one depends on the nature
of the noise. In a so-called \emph{pushed} regime \cite{Kueh} (for example in presence of a genetic drift
term) the effect can be surprisingly strong also for small noise, as demonstrated in the seminal
work by Mueller, Mytnik and Quastel \cite{MuellerMytnikQuastel}. In our case,
the nonlinearity is smooth and concave: We are in the
\emph{pulled} regime, where the effect of noise is weaker and most importantly
the speed of the wave is governed by the linearisation of the equation near $
u = 0 $.

Now, the Lyapunov exponent of the linearised equation is in turn
described the long-time behaviour of the dual process and the wave speed can be
easily rewritten as the speed of the rightmost particle of the dual. This
correspondence is very well understood: For the deterministic equation the dual
is given by a Branching Brownian Motion (BBM)
\cite{Ikeda3,McKean}. In the BBM each
particle moves as an
independent Brownian motion and branches into two
identical offspring at a constant rate $\mathbf r$
(see~\cite{Ikeda1,Ikeda2,Ikeda3}). 
In our setting, the dual is given by a \emph{coordinated branching Brownian motion}
(CBBM), see also the discussion of existing literature below. The main difference
with respect to the BBM is that particles tend to
reproduce simultaneously rather than independently: If $n$
particles are alive at a given time, for any $1 \leq k \leq n$, any $k$-tuple of
particles decides to simultaneously produce one offspring per particle at rate
$\int_{[0,1]} y^k (1-y)^{n-k} \frac{1}{y} \mathfrak R(\d y)$.

To study the speed of the rightmost particle of the CBBM, we consider a general 
approach developed by Kyprianou and Englaender
\cite{Kyprianou, Englander} (cf. also the references therein), which uses a
martingale argument to study the local survival of branching (or super-)
processes. As a rule-of-thumb, it states that the speed of the rightmost
particle equals $ \sqrt{2 \lambda } $, where $ \lambda $ is the Lyapunov exponent
of the underlying system.
A subtlety when applying their argument in our setting is that the \emph{quenched} (so conditional on the realisation of the jumps) growth rate of the number of particles of the CBBM and
\emph{annealed} growth rate (of the expected number of particles of the CBBM, where the expectation is taken also over the random environment) might differ: An environment exhibiting such
a behaviour is called \emph{strongly catalytic}. While in the classical BBM the almost sure and the expected growth rate coincide, our coordinated process is strongly catalytic. This imposes a
fundamental new challenge to our analysis, as the correct prediction for the wave
speed appears by formally using the rule-of-thumb with the \emph{quenched} Lyapunov exponent $ \lambda =
\frac{\mf{s}^{2}}{2} $, as in \eqref{eqn:intro-wave-speed} (note that this
corresponds to the growth rate of the total mass of the CBBM), while an attempt to
make the martingale argument rigorous breaks down, because the gap between
quenched and annealed implies that now the martingales at hand are not uniformly
integrable. 
We observe that this issue is related to the gap between the
stochastic and the deterministic equation we already addressed, as the annealed
Lyapunov exponent is given by $ \mathbf{r} = \mf{R}([0, 1]) $ (which is a
strict upper bound on our speed, unless $ \mf{R} = \mf{r} \delta_{0} $).

Our approach to overcome this problem is to distinguish
between `large reproduction events' in which individuals participate in an
event with probability $y \in (\delta,1]$, and `small reproduction events' where
individuals participate with probability $y \in (0,\delta]$, for some $\delta
\in (0,1]$. Then we proceed to obtain upper and lower bounds on the speed $
\mf{s} $, which depend on $ \delta $ but converge to the correct speed as $
\delta \downarrow 0$. For the upper bound, we use a quenched dual, where we
condition on the location and impact of `large' reproduction events. We then
use the martingale argument outlined above to deal with small jumps, which now
affect the speed only by \color{black} a factor $ \mf{R}((0, \delta]) $. Instead, for large
reproduction events we use time changes and a ``channeling'' argument based on
elementary large-deviation estimates to obtain the expected contribution to the
speed. For the lower bound we use comparison to remove the mass of $
\mf{R} $ in the interval $ (0, \delta] $ and then can proceed with similar
calculations to those we use the upper bound.

Overall, the novelty of our work consists in quantifying the effect of extreme
selective events on the speed of invasion of the favoured of two genetic types.
In future, we hope to extend these results to a much broader class of models,
potentially including dormancy \cite{BHN20, BJN21}, mutation, genetic drift and spatially
localized selective events \cite{BEV}. Finally, we note that we only consider
the highest order (linear) term in the wave speed.
For the original Fisher-KPP equation many, more refined results are available
\cite{CHL,Lalley-Sellke} and would be interesting to extend to the present
setting. Similarly, the existence of a (generalized) travelling wave (and not
  just the speed of propagation) is left open.

\subsection{Related literature}\label{sec-rareselection}

Recently, the study of the effect of extreme selective or
reproductive events on evolutionary models has seen a flurry of activity. 
An archetypal non-spatial model for such an evolution is the $ \Lambda $--coalescent, in
which a measure $ \Lambda $, corresponding to our $ \mf{R} $, determines the
proportion of individuals participating in a merger event \cite{pitman,
Sagitov}: see also \cite{Griffeath} for one of the first examples of
coordinated reproduction in the context of contact
processes, and \cite{GKT} (and the references therein) for a general
framework regarding coordination in reproduction, death and migration. In the
study of non-spatial models, extreme selection and reproduction events -- which are
in correspondence via duality -- have been recently addressed by
\cite{BCM, CV19, GCSW}. In the study of spatial models such as superprocesses, the effects of selection have been analysed for example in \cite{PR, KR}.

For example, Cordero and V\'echambre \cite{CV19} derive an analogue of our
equation, with genetic drift and no spatial component, as the scaling limit of a
microscopic particle system and study its long-time
behaviour (similar scaling limits and results have been obtained in \cite{BCM,
GCSW} in related models). Although these works do not consider our spatial setting, they share key
aspects of our approach. In particular, the use of duality and the study of the
long-time behaviour of the processes through conditioning or averaging
over the environment are essential in our arguments.

In more detail, the most well-known population-genetic example of moment duality is the one between the Wright--Fisher diffusion (without selection) and the Kingman coalescent. The Wright--Fisher diffusion appears as the scaling limit of the relative frequency of a neutral allele in the Wright--Fisher model, which runs forwards in time, whereas the Kingman coalescent describes the ancestry of a sample of the haploid and asexually reproducing individuals of the Wright--Fisher model backwards in time. While the introduction of selection into the forwards-in-time process is straightforward, the existence of a moment dual was not known before Krone and Neuhauser \cite{KN97} introduced this process, called the \emph{ancestral selection graph}. In this graph, while random genetic drift still leads to mergers of ancestral lines in the dual, selection makes ancestral lines branch into multiple potential parents. For example, if the model consists of just one weak and one strong allele, both potential parents of a particle with a weak allele type must be themselves of a weak type. The moment duality between the classical BBM and the solution to the classical FKPP equation can be interpreted similarly; this is a spatial model without random genetic drift, where the forwards-in-time process is deterministic, and the Brownian particles of its dual exhibit branching only. If one introduces a rare selection governed by a Poisson point process just as in Equation~\eqref{eqn:fkpp}, then the corresponding part of the dual process will be governed by the same Poisson point process. Similarly to the Wright--Fisher diffusion with selection, if $(t,y)$ is a point belonging to the Poisson point process, then the forwards-in-time interpretation of the model is that at time $t$ a fraction $y$ of individuals, chosen uniformly at random, are participating in a large selective resampling event. On the other hand, as in the the ancestral selection graph, backwards-in-time this corresponds to a large scale branching event, in which each particle  participates with probability $y$.

As also mentioned in \cite{CV19,GCSW}, examples of experimental studies on rare selection can be found in \cite{lizard,antibiotic}. In \cite{lizard}, lizards with long fingers can hold on stronger and thus avoid being blown away whenever their habitat is hit by a hurricane, which provides them a strong but temporary selective advantage. Further, \cite{antibiotic} compares different antibiotic treatment strategies against a bacterial population. Here, the analogue of a continuously present but weak selective pressure is a constant administration of the antibiotic in low concentration dosage, while rare and strong selective events correspond to a less frequent inoculation with higher dosages (possibly of varying concentration and at random times).\color{black}

We also note that extreme evolutionary events in a spatial \color{black} setting have received much attention over the
past years in relation to the study of spatial $ \Lambda$--Fleming--Viot
(SLFV) models introduced by Barton, Etheridge and V\'eber \cite{BEV}. Unlike
our equation, in this class of processes reproductive events are localized in
space, which is a natural assumption and an
interesting direction for future extensions of our result.

After completion of the present paper, we learned that the speed of the rightmost particle of the CBBM can be computed also via the results of~\cite{MM18} on branching random walks in a time-inhomogeneous random environment, using different tools. 

Together with our results on well-posedness of \eqref{eqn:fkpp} its duality with respect to the CBBM, this provides an alternative proof of the wave speed of~\eqref{eqn:fkpp}.

\subsection{Structure of the paper}
This article is divided as follows. In Section~\ref{sec:model} we present our model and in
Section~\ref{sec:results} we state our main results, along with the crucial
points of their proofs.
The technical details of the proofs are carried out in the rest of the paper.
 In Section~\ref{sec:existenceduality} we
prove the existence and uniqueness of strong solutions to~\eqref{eqn:fkpp} as
well as (quenched and annealed) duality.
These results do not come as a surprise, but require a proof and a precise
statement. Section~\ref{sec:wavespeed} is devoted to upper and lower bounds on the wave
speed via quenched duality arguments.

\subsection{Notations}

We write $ \NN = \{ 1, 2, \dots \} $ and denote $[n]=\{1,\ldots,n\}$ for any $n
\in \N$.  Furthermore, let $\mM$ be
the space of finite positive Borel measures on $[0,1]$ with the topology
of weak convergence. For a set $ \mX $ and two functions $ f, g \colon \mX 
\to \RR $ we write $ f \lesssim g $ if there exists a constant $ c > 0 $ such
that $ f(x) \leqslant c g(x) $ for all $ x \in \mX $. If the constant $c$ depends on some parameter $\vartheta$ we write $f \lesssim_\vartheta g$. We further denote with $
C^{k}_{b}(\RR; \mO) $ (for $ k \in \NN \cup \{\infty\} $ and any target set $\mO \subseteq \RR$) the space of bounded and  $ k $ times differentiable 
functions $ \varphi \colon \RR \to \mO $ with continuous and bounded derivatives. Similarly, for $ \gamma \in (0, \infty) \setminus \NN $ we define $ C^{\gamma}_{b} $ to
be the space of bounded and $ \lfloor \gamma \rfloor$--times differentiable functions
with $ \gamma - \lfloor \gamma \rfloor- $Hölder continuous and bounded derivatives.

Finally, with $ C_{\mathrm{loc}} (\RR; \mO)$ we denote the space of continuous
(and not necessarily uniformly bounded) functions with values in $\mO$. When $\mO=\RR$ we may drop the dependence on it in the notation. The spaces $C^k_b$ and $C^\alpha_b$, for $k \in \NN, \alpha \not\in \NN$ come equipped, respectively, with the norms
\begin{align*}
    \|\varphi\|_{\infty} = \sup_{x \in \RR} |\varphi(x)|\;, \quad  \| \varphi \|_{C^k_b} = \sum_{i = 0 }^k \| \partial_x^i \varphi \|_{\infty} \;,
\end{align*}
and
\begin{align*}
     \| \varphi\|_{C^\alpha_b} =  \sum_{i=0}^{\lfloor \alpha \rfloor } \| \partial_x^i \varphi \|_{\infty} + \sup_{x \neq y} \frac{|\varphi(x) - \varphi (y)|}{|x-y|^{\alpha- \lfloor \alpha \rfloor}}\;.
\end{align*}
Moreover, for any polish
space $ E $ we indicate with $ \DD ([0, \infty) ; E) $ the space of
c\`adl\`ag paths with values in $E$ endowed with the Skorokhod topology
(similarly for $ [0, \infty) $ replaced by some finite interval $ [0, T] $).

\section{Setting and main results}\label{sec:modelmainresults}

\subsection{The model}\label{sec:model}

The main object of
interest in this work is the wave speed $ \mf{s} $ of the solution to
\eqref{eqn:fkpp}. It will be convenient to consider the following class of initial
conditions, for any $ \alpha \in (0, 1) $ 
{
\begin{align*}
C^{\alpha}_{0, 1} = \big\{  u \in C^{\alpha}_b(\RR; [0, 1]) & \text{ such that } \overline{\{ x  \
\colon \ u(x) \not\in \{ 0, 1 \} \}} \text{ is compact} \\
& \text{ and such that } \lim_{x \to - \infty } u(x) =1, \quad \lim_{x \to \infty} u(x) = 0  \big\}\;,
\end{align*}}
for which the wave speed is naturally defined below. 
\begin{definition}\label{def:wave-speed}
  We say that $ \mf{s} \in \RR $ is the wave speed associated to
  \eqref{eqn:fkpp} if for any $ \alpha \in (0,1) $ and all $ u_{0} \in
C^{\alpha}_{0, 1} $ the following hold.
  \begin{enumerate}
    \item For every $ \lambda > \mf{s} $ and any $ x \in \RR $, we have $ \lim_{t \to \infty}
      u_{t}(x + \lambda t) = 0$ in probability.
    \item For every $ \lambda < \mf{s} $ and any $ x \in \RR $, we have $ \lim_{t \to \infty}
      u_{t}(x +  \lambda t ) =1 $ in probability.
  \end{enumerate}
\end{definition}
\begin{remark}\label{rem:heaviside}
  Our initial conditions are chosen to be H\"older continuous, to simplify the
  statements that will follow. We could consider also discontinuous initial
  data, e.g.\ {$ u_{0}(x) = \mathds 1_{(-\infty, 0]}(x) $,} at the cost of introducing
  blow-ups at time \( t=0\) in the solution theory for \eqref{eqn:fkpp}.
\end{remark}
The study of the wave speed of the solution to \eqref{eqn:fkpp} passes through the analysis of
its dual process, which consists of a system of Brownian
motions with coordinated branching that run backwards in time and roughly represent the genealogy of types of a sample of particles. In the dual process the parameter $ y $ interpolates between no coordination ($
y=0 $, so all particles act independently) and full coordination ($y=1
$, so all particles reproduce at once). In this backwards (or dual) picture the
measure $ \mf{R} $ captures the reproduction rate.

\begin{notation} To describe the state space of our particle systems we
introduce the set $$ \mP = \bigsqcup_{n \in \NN} \RR^{n}\;.$$ 
Then to every point $ \mathbf{x} $ of $ \mP $ we can associate a length $ \ell
(\mathbf{x}) = n \iff \mathbf{x} \in \RR^{n}. $ In particular, $ \mP $ is a Polish space with the
distance $ d (\mathbf{x}, \mathbf{y}) = | \ell(\mathbf{x}) -
\ell(\mathbf{y}) |+ \| \mathbf{x}- \mathbf{y} \| \mathds 1_{\{ \ell(\mathbf{x}) =
\ell(\mathbf{y}) \}} $, where $ \| \cdot \| $ indicates the Euclidean norm. To concisely express
our duality formulas, let us introduce the following notation
\begin{align*}
\varphi^{\mathbf{x}} = \prod_{i = 1}^{n} \varphi (x_{i}), \qquad \forall \varphi \in
C_{\mathrm{loc}}(\RR), \quad \mathbf{x}= (x_{1}, \dots , x_{n}) \in \mP.
\end{align*}
In addition, for $ \mathbf{x}, \mathbf{y} \in \mP $ we write the concatenation
\begin{align*}
\mathbf{x} \sqcup \mathbf{y} = (x_{1}, \dots, x_{\ell(\mathbf{x})},
y_{1}, \dots, y_{\ell(\mathbf{y})}) \in \mP.
\end{align*}
\end{notation}

The way in which we use duality requires the introduction of an additional
parameter $ \delta \in (0, 1] $. We will then consider a dual
\emph{conditional} on jumps with impact $ y >\delta $. 
We start by distinguishing small from large jumps with the following notation.

\begin{definition}\label{def:R-delta}
For any $ \delta \in (0, 1] $ and $ \mf{R} \in \mM ([0, 1]) $ define
\begin{align*}
\mf{R}_{\delta}^{-}(A) = \mf{R}(A \cap [0, \delta]) \;, \qquad \mf{R}_{\delta}^{+} (A) = \mf{R} (A \cap (\delta, 1]) \;, \qquad \mf{R}_\delta = (\mf{R}_\delta^- , \mf{R}_\delta^+)\;.
\end{align*}
In general, we call \textbf{\emph{compatible}} with $ \delta $ any ordered pair of measures $
\mu_{\delta} = (\mu_{\delta}^{-}, \mu_{\delta}^{+}) \in
\mM^{2} $ with support in $ [0, \delta] $ and $ [\delta, 1] $ respectively with
$ \mu_{\delta}^{+}(\{ \delta \}) =0$.
Finally, for any compatible measures $ \mu_{\delta} $ we introduce the Poisson
point process with intensity $ \ud t \otimes \frac{1}{y} \mu_{\delta}^{+} (\ud y)$:
\begin{align}\label{eqn:def-skeleton}
  \mS_{\delta} = \{ (t_{j}, y_{j}) \}_{j \in \NN}
  \subseteq  [0, \infty) \times (\delta,1] \;,
\end{align}
which is characterised by the fact that $ 0 < t_{1} < \dots < t_{j} < t_{j+1}$,
and $ t_{j} \uparrow \infty $, and is linked to the Poisson random measure
\begin{align*}
  \mR_{\delta}^{+} (\ud t, \ud y) = \sum_{j \in \NN} \delta_{t_{j}}(\ud t)
  \delta_{y_{j}} ( \ud y) \;.
\end{align*}
\end{definition}
We observe that formally we can rewrite the noise $ \mR $  in \eqref{eqn:fkpp}
as 
\begin{align*}
\mR = \mR_{\delta}^{-} + \mR_{\delta}^{+} \; ,
\end{align*}
with $ \mR_{\delta}^{-}$ a Poisson random measure with intensity $ \ud t \otimes \frac{1}{y} 
\mf{R}_{\delta}^{-}(\d y) \color{black}$. To be precise, $ \mR $ is in general not a 
measure, but can only interpreted when integrated against functions that vanish
near $ y=0 $ sufficiently fast. More precisely, $\mathcal R_\delta^-$ is
associated with a Poisson point process $\{ (s_j,z_j) \colon i \in I \}$
of intensity $\d t \otimes \frac{1}{y} \mf R(\d y)$ on $[0,\infty) \times
(0,\delta]$, with countable index set $I \subseteq \NN$ (so both the points
$(s_j,z_j)$ and the index set $ I $ are random). Then for measurable functions $f \colon [0,\infty) \times (0,\delta] \to \R$ satisfying
\[ \int_{[0,\infty) \times (0,\delta]} \min \{ |f(t,y)|, 1 \} \frac{1}{y} \mf
R(\d y) \ud t < \infty \;, \numberthis\label{eqn:integrabilitycondition} \]
the integral
\[ \int_{[0,\infty) \times (0,\delta]} f(t,y) \mathcal R(\d t, \d y):=\sum_{j\in I} f(s_j,z_j) \numberthis\label{eqn:Poissonintegral} \]
is almost surely finite, as discussed in the context of Campbell's theorem in \cite[Section 3]{K93}.
We can now introduce the dual process to \eqref{eqn:fkpp} \emph{conditional} on the
realisation of $ \mS_{\delta}$. We highlight that the FKPP equation and its conditional dual share the
same jump times $ \mS_{\delta} $ (to be precise, the dual process
\emph{may} jump at a time contained in $ \mS_{\delta} $ but does not
necessarily do so), whereas they do not share the jump times associated to
smaller impacts.
\begin{definition}[$\mu_\delta$--CBBM]\label{def:rs-cbbm}
For any $ \delta \in (0,1] $ and any couple $
\mu_{\delta} =(\mu_{\delta}^{-}, \mu_{\delta}^{+}) $ compatible with $ \delta
$, let $ \mS_{\delta} $ be the
Poisson point process defined by \eqref{eqn:def-skeleton}. We say
that 
 $ (\mathbf{C}_{t})_{t \geqslant 0} $ is a $
\mu_{\delta}$--coordinated branching
Brownian motion ($ \mu_{\delta}$--CBBM) with initial condition $
\mathbf{C}_{0} = \mathbf{x} \in \mP $ if, conditional on the realisation 
of $ \mS_{\delta} $, the process $ \mathbf{C}_{t} $ is a $ \mP$-valued Markov
process with the following dynamics:
\begin{enumerate}
    \item \textbf{Diffusion.} Let $ \mathbf{C}_{t} = (x_{1}, \dots,
x_{n}) $ at time $ t > 0 $. Then each individual $ x_{i} $ moves in $\R$ according to a Brownian
motion, independent of all other individuals, until one of the following two
jumps occur. 
    \item \textbf{Large reproduction events.}  For every $ j \in \NN $, assume that at time $ t_{j} $ (of the
Poisson point process $ \mS_{\delta} $) there are currently $n$ individuals $
\mathbf{C}_{t_{j} } =
      (x_1,\ldots,x_n) \in \RR^{n}$. \\
Then we observe one of the following
transitions, for any subset $\mathcal I \subseteq [n]$:
      $$ (x_i)_{i \in [n]} \to (x_i)_{i \in [n]} \sqcup (x_{i})_{i \in \mI} \in
      \RR^{n + | \mI |} \quad \text{ with probability } \quad
      y^{|\mathcal I|}_{j} (1-y_{j})^{n-|\mathcal I|}.$$
\item \textbf{Small reproduction events.}
Assume that at time $ t \geqslant 0 $ there are currently $n$ individuals $
\mathbf{C}_{t} = (x_1,\ldots,x_n) \in \RR^{n}$. Then for any subset $\emptyset
\neq \mathcal I \subseteq [n]$ we have the following coordinated transition:
      $$ (x_i)_{i \in [n]} \to (x_i)_{i \in [n]} \sqcup (x_{i})_{i \in \mI} \in
      \RR^{n + | \mI |} \quad \text{ at rate } \quad
      \int_{[0, \delta]} y^{|\mathcal I|} (1-y)^{n-|\mathcal I|} \frac{1}{y}
      \mu_{\delta}^{-} (\d y) \;.$$  
\end{enumerate}
\end{definition}
We observe that for $ \delta =1 $ we have $ \mS_{\delta} = \emptyset $. Then the
dynamics of the CBBM do not have a discrete reproduction component, and in
this case the process is the \emph{unconditional} dual of \eqref{eqn:fkpp}. The
duality between the CBBM and the FKPP equation will be established in Proposition~\ref{prop:cond-dual}.

The necessity of dealing with $ \delta \in (0,1) $ (and in particular, we will
eventually consider the limit $ \delta \to 0 $) is forced upon us to capture
the exact wave speed of \eqref{eqn:fkpp}. In fact the martingale problem for the
$ \mf R_1 $--CBBM (or alternatively, \cite[Lemma 3]{GKT}) implies the
following (in fact $ e^{- \mathbf{r} t} I_{t} $ is a martingale, although it is
in general \emph{not} uniformly integrable).

\begin{proposition}[Invariance of expectation]\label{prop-invexp}
Let $\mf{R} \in \mM$ be any measure and $\mathbf C=(\mathbf C_{t} )_{t \geq 0}$ be an $ \mf{R}_{1} $--CBBM and write
$I_t = \ell (\mathbf{C}_{t}) $ for the total number of particles at
time $ t \geqslant 0 $. Then 
\begin{align}\label{eqn:grwth-expct}
 \E[I_t]= I_0\mathrm e^{\mathbf{r} t}\;, \quad \text{ with } \quad \mathbf{r} =
\mf{R}([0, 1]) \;.
\end{align}
\end{proposition}
In the present case, in which the nonlinearity
in \eqref{eqn:fkpp} is concave, the wave speed is determined by the growth of
the linearisation near $ u=0 $ of the equation: this is referred to as the
\emph{pulled} regime \cite{Kueh}.
Moreover, the growth of the linearisation is roughly
equivalent to that of the dual process.
On the other hand, Jensen's
inequality guarantees that the speed of the expected value of the solution to \eqref{eqn:fkpp} is strictly slower than in the
classical case (with same total mass for the reproduction), since 
\begin{align*}
  \partial_{t} \EE [u_{t}] \leqslant \frac{1}{2} \Delta \EE [u_{t}] + \mathbf{r}
  \EE[u_{t}](1 - \EE[u_{t}]) \;,
\end{align*}
where we used that $ M_{t}^{f} = \int_{0}^{t} \int_{0}^{1} y f_{s} \mR ( \ud s, \ud y ) -
\int_{0}^{t}  \mathfrak R((0,1])  f_{s} \ud s $ is a martingale for bounded, adapted $
f $. As an educated guess, one can think that the speed of the expected value of the solution is the same as the wave speed of the solution itself, and Theorem~\ref{thm:wave-speed} below shows that this is indeed true. \color{black}
Hence we are faced with an apparent conundrum, as the speed predicted by
\eqref{eqn:grwth-expct} is exactly the deterministic (annealed) one, which we
now know to be incorrect. 

The issue is that the coordinated process $ \mathbf{C}_{t} $, unlike the branching
Brownian motion, is strongly catalytic (apart from the case $\mf{R} =c \delta_0$, $c>0$, in which the two processes coincide)\color{black}: Namely its almost sure growth rate is strictly
smaller than its annealed growth rate, captured by \eqref{eqn:grwth-expct}. For
this reason, classical martingale arguments do not work directly. 

Our approach is therefore to use
the conditioning as a way to obtain the almost sure growth rate. As usual for
Poisson point processes, one has to take particular care of the small jumps:
for this reason we consider a fixed parameter $ \delta > 0 $.
Small jumps are then dealt with via the
argument we just explained, through Jensen's inequality and martingales. This delivers a wrong estimate, but now with an
error of order $ \mO(\delta) $, in such a way that as $ \delta \to 0 $ we obtain the correct
speed.

\subsection{Main results}\label{sec:results}
Now we are ready to present our main results. We start by proving
well-posedness of \eqref{eqn:fkpp}.

\begin{theorem}\label{thm:existence-uniqueness}
Fix any $ \mf{R} \in \mM $ and let $ (\Omega, \mF, \PP) $ be a probability space supporting a Poisson point
process $ \mS $ on $ [0, \infty)\times  (0, 1] $ with intensity measure $ \ud t \otimes \color{black}
\frac{1}{y} \mf{R}(\ud y)  $. Let $ \mF_{t}$ be the
right-continuous filtration generated by $ \mS^{t} = \mS \cap \left( [0,
t] \times (0, 1] \right)$. For any $\alpha \in (0, 1) $ and any
initial condition $ u_{0}$ in $ C^{\alpha}_{b} (\RR; [0, 1]) $ there exists a
unique (up to modifications on a nullset) adapted process $$ u \colon \Omega \to
\DD( [0, \infty) ;  C^{\alpha }_{b}([0, 1]; \RR) ) $$ that solves \eqref{eqn:fkpp} on
$ [0, \infty) \times \RR $ (with the derivatives interpreted in the sense of
distributions and the integral against $ \mR $ interpreted 
 in the sense of sums over Poisson point processes, cf.~\eqref{eqn:Poissonintegral} for $\delta=1$) with $ u (0, \cdot) = u_{0}(\cdot) $.
\end{theorem}
This result is a consequence of Proposition~\ref{prop:existence}. For the
solution we just constructed we can describe the wave speed as follows.

\begin{theorem}\label{thm:wave-speed}
  For every $ \mf{R} \in \mM,\, \alpha \in (0, 1) $ and any initial condition
$ u_{0} \in C^{\alpha}_{0,1} \color{black}  $, the solution $ u $ to the FKPP equation
\eqref{eqn:fkpp} with initial condition $ u_{0} $ (as in
Theorem~\ref{thm:existence-uniqueness}) has wave speed $ \mf{s} > 0 $ in the sense
  of Definition~\ref{def:wave-speed} given by
\begin{equation}\label{eqn:def-speed}
  \frac{\mf{s}^{2}}{2}  = \int_{[0, 1]} \log{(1 + y)} \frac{1}{y}
  \mf{R}(\ud y) \;. 
\end{equation}
\end{theorem}
Again, we follow the convention $\int_{\{0\}} \log(1+y) \frac{1}{y} \mf R(\ud
y) = \mf{R}(\{ 0 \}) = \mf{r}$.


\begin{proof}
We follow two different arguments for the lower and upper bounds to the
wave speed (the two conditions in Definition~\ref{def:wave-speed}).

\textit{Step 1.} Let us start with the upper bound, so fix any $ \lambda > \mf{s} $. Our aim will be
to prove that for any $ x \in \RR $
\begin{align*}
\lim_{t \to \infty} \EE u_{t}(x  + \color{black} \lambda t) = 0\;,
\end{align*}
which implies the required convergence in probability. For this purpose
consider $ \delta \in (0, 1) $ and define $ \mf{R}_{\delta}=  (\mf{R}_{\delta}^{-},
\mf{R}_{\delta}^{+}) $ as in Definition~\ref{def:R-delta}, associated to the
decomposition $ \mR = \mR_{\delta}^{-} +
\mR_{\delta}^{+} $, where $
\mR_{\delta}^{+} $ is the random measure associated to the Poisson point process $ \mS_{\delta}
$ with intensity $  \ud t \otimes  \frac{1      }{y}\mf{R}^{+}_{\delta}(\d y) \color{black}$. Then let $ \EE^{\delta} $ indicate the expectation conditional on $
\mS_{\delta} $, namely $$ \EE^{\delta}[f] = \EE[f \vert \mS_{\delta}]\;. $$ Since
$ u_{t}(x) $ takes values in $ [0, 1] $ by dominated convergence it thus
suffices to prove that {if $\delta = \delta (\lambda) \in (0,1)$ is sufficiently small, then} $ \PP $--almost surely
\begin{align*}
\lim_{t \to \infty} \EE^{\delta}u_{t}(x  + \color{black} \lambda t) = 0 \;.
\end{align*}
Here we use the conditional duality of Proposition~\ref{prop:cond-dual} to bound
\begin{align*}
\EE^{\delta} [1- u_{t}(x  + \color{black} \lambda t)] = \EE^{\delta} \left[ (1 -
u_{0})^{\mathbf{C}_{t}^{\mathbf{x}(\lambda, t)}} \right] \;,
\end{align*}
where $ \mathbf{C}_{t}^{\mathbf{x}(\lambda, t)} $ is an $
\mf{R}_{\delta}$--CBBM as
in Definition~\ref{def:rs-cbbm}, started in $  \mathbf{x}(\lambda, t) = x +
\lambda t \in
\RR^{1} $. Now since $ u_{0} \in C^{\alpha}_{0, 1} $ there exists an $ a \in
\RR $ such that $ u_{0}(x) = 0 $ for all $ x \geqslant  a $. In particular
\begin{align*}
\EE^{\delta} (1 - u_{0})^{\mathbf{C}_{t}^{\mathbf{x}(\lambda, t)}} \geqslant \EE^{\delta}(  \mathds 1_{[a, \infty)} \color{black})^{\mathbf{C}_{t}^{\mathbf{x}(\lambda, t)}} = \PP^{\delta}( \mathbf{S}_{t} \leqslant  -a +
\lambda t + x) \;,
\end{align*}
where $ \mathbf{S}_{t} =\max \mathbf{C}_{t}^{0} $ is the rightmost particle of
an $ \mf{R}_{\delta}$--CBBM $ \mathbf{C}^{0}_{t} $ started in $ \mathbf{x} = 0 \in \RR^{1} $ {(note that by symmetry $\PP^\delta(\max \mathbf{C}_t^0 \leq c) = \PP^\delta (\min \mathbf{C}_t^0 \geq -c) = \EE^{\delta}( \mathds 1_{[-c, \infty)} )^{\mathbf{C}_{t}^{0}}$)}.
Hence it suffices to show that for any $ x_{0} \in \RR $
\begin{align*}
\lim_{t \to \infty} \PP^{\delta} (\mathbf{S}_{t} > \lambda t + x_{0}) =0 \;.
\end{align*}
This claim follows from Proposition~\ref{prop:manytoone}, up to choosing
$ \delta $ sufficiently \color{black} small so that for $ \mf{c}_{\delta} $ as in
\eqref{eqn:logspeed}
\begin{align*}
\mf{s} \leqslant  \sqrt{2 \mf{c}_{\delta}} < \lambda \;.
\end{align*}

\textit{Step 2.} Let us now pass to the lower bound. That is, choose $ \lambda <
\mf{s} $ and, similarly to above, let us prove that
$\lim_{t \to \infty} \EE u_{t}(x  + \color{black} \lambda t) = 1$.
As before we can fix $\delta \in (0, 1) $, so that it suffices to prove that
$ \PP $--almost surely
\begin{align*}
\lim_{t \to \infty} \EE^{\delta} u_{t}(x  + \color{black} \lambda t) = 1 \;.
\end{align*}
Then by the duality of Proposition~\ref{prop:cond-dual} we have
\begin{align*}
\EE^{\delta} (1- u_{t})(x  + \color{black} \lambda t) = \EE^{\delta} (1 -
u_{0})^{\mathbf{C}_{t}^{\mathbf{x}(\lambda, t)}} \leqslant \EE^{\delta} (1 -
u_{0})^{ \underline{\mathbf{C}}_{t}^{\mathbf{x}(\lambda, t)}}\;.
\end{align*}
Here $ \mathbf{C}^{\mathbf{x}(\lambda, t)}_{t} $ is an $ \mf{R}_{\delta} $--CBBM,
with $ \mf{R}_{\delta} =( \mf{R}_{\delta}^{-},  \mf{R}_{\delta}^{+}) $ started in $ \mathbf{x}(\lambda, t) = x  + \color{black} \lambda t \in
\RR^{1}$, and $ \underline{\mathbf{C}}^{\mathbf{x}(\lambda, t)}_{t} $ is an $
\underline{\mf{R}}_{\delta} $--CBBM associated to  compatible \color{black} measures $ \underline{\mf{R}}_{\delta} =
(\mf{r} \delta_{0}, \mf{R}^{+}_{\delta}) $, started in $ \mathbf{x}(\lambda, t) $. Then
we can use that by definition
\begin{align*}
  \mf{r} \delta_{0} \leqslant \mf{R}_{\delta}^{-} \;,
\end{align*}
in the sense of measures, so that we can couple $ \mathbf{C}_{t} $ and $
\underline{\mathbf{C}}_{t} $ in such a way that $ \underline{\mathbf{C}}_{t} \subseteq
\mathbf{C}_{t} $, which implies the desired estimate. In particular it now
suffices to prove that 
\begin{align*}
  \lim_{t \to \infty} \EE^{\delta} (1 - u_{0})^{
  \underline{\mathbf{C}}_{t}^{\mathbf{x}(\lambda, t)}} =0 \;.
\end{align*}
Again, we find $ b \in \RR $ such that $ u_{0} (x) = 1 $ for all $ x
\leqslant b$, so that
\begin{align*}
 \EE^{\delta} (1 - u_{0})^{ \underline{\mathbf{C}}_{t}^{\mathbf{x}(\lambda, t)}}
\leqslant \PP^{\delta}( \underline{\mathbf{S}}_{t}  \leqslant- b +\color{black} \lambda t+\color{black}x) \;,
\end{align*}
where $ \underline{\mathbf{S}}_{t} = \max \underline{\mathbf{C}}_{t}^{0} $, the
latter being a $ (\mf{r} \delta_{0}, \mf{R}_{\delta}^{+}) $--CBBM started in $
\underline{\mathbf{C}}^{0}_{0} = 0 \in \RR^{1}. $ Now by
Proposition~\ref{prop:rightmost-speed} we know that if $ \delta \in (0, 1) $ is
chosen to be sufficiently small such that for $ \underline{\mf{c}}_{\delta} $ as in
\eqref{eqn:c-delta-low}
\begin{align*}
\lambda < \sqrt{2 \underline{\mf{c}}_{\delta}} \leqslant \mf{s}\;,
\end{align*}
then $ \lim_{t \to \infty}
\PP^{\delta}( \underline{\mathbf{S}}_{t}  \leqslant -b + \color{black} \lambda t+\color{black}x) =  0 \color{black}$, $ \PP $--almost
surely. The proof is concluded.
\end{proof}

\section{Existence and duality}\label{sec:existenceduality}
This section is devoted to proving existence and uniqueness of solution to
\eqref{eqn:fkpp}, as well as duality. We will first construct unique solutions
and observe that they satisfy a certain martingale problem. Then we use the
martingale problem to establish duality.

\subsection{Existence and uniqueness}

Let us start by defining the generator associated to the nonlinearity of
\eqref{eqn:fkpp}. To be precise, the first definition will be associated to the
space-independent equation. The extension to the spatial case passes through
cylinder functions, as explained in the subsequent definition of martingale
solutions. Here the set $ C_{c}^{\infty} $ indicates the space of smooth
functions with compact support on $ \RR$. Throughout these construction, we
recall that in Definition~\ref{def:R-delta} we have divided
\[ \mf{R} = \mf{R}_{\delta}^{-} + \mf{R}_{\delta}^{+}\;. \] 
\begin{definition}
Fix any $ \mf{P} \in \mM$. For any $n$-tuple of 
smooth functions $ \boldsymbol \varphi=(\varphi_1,\ldots,\varphi_n) \in (C_{c}^{\infty})^n $ and an $ F \in
C^1_b(\RR^n) $ define the cylinder function
\begin{align*}
C_{\rm{loc}} (\RR) \ni u \mapsto F_{\boldsymbol{\varphi}}(u) = F (\langle u , \varphi_1 \rangle,\ldots,\langle u , \varphi_n \rangle)\;,
\end{align*}
where $\langle u, \varphi_i \rangle = \int_{\RR} u \varphi_i \ \ud x$, \color{black}
and for any such $ F_{\boldsymbol{\varphi}} $ we define
the generator $$ \mL^{\mf{P}}
F_{\boldsymbol{\varphi}} \colon C_{\mathrm{loc}}(\RR) \to \RR\color{black} $$ as follows
\begin{align*}
\mL^{\mf P} (F_{\boldsymbol{\varphi}})(u) & = \int_{(0,1]} \left\{ F_{\boldsymbol{\varphi}}(u + y u(1-u)) - F_{\boldsymbol{\varphi} }
(u) \right\} \frac{1}{y} \mf{P} (\ud y)\;, 
\end{align*}
for all $u \in C_{\mathrm{loc}}(\RR)$.
\end{definition}
We observe that the integral is well-defined at $ y=0 $ for every $ u $, since 
\begin{align*}
F_{\boldsymbol{\varphi}} & (u + y u (1- u)) - F_{\boldsymbol{\varphi}} (u) \\
& = \sum_{i=1}^n \partial_i F (\langle u, \varphi_1 \rangle,\ldots,\langle u, \varphi_n \rangle) y \langle u (1-u), \varphi_i \rangle + o(y) \;,
\end{align*}
for $ y \to 0 $ as $ F \in C^1_b $.
\begin{remark}\label{rem:bd-gen}
In particular, we can bound for any $F \in C^1_b$ and $\boldsymbol\varphi=(\varphi_1,\ldots,\varphi_n)$:
\begin{align*}
| \mL^{\mf{P}}(F_{\boldsymbol{\varphi}}) (u) | \lesssim_{\| F \|_{C^1_b}, \sum_{i=1}^n \| \varphi_i \|_{L^1}}
\mf{P}((0, 1]) \;,
\end{align*}
where $\|\varphi_i\|_{L^1}= \int_{\RR} |\varphi_i(x)| \ud  x$.
\end{remark}
Next we give a precise definition of martingale solutions to the stochastic
FKPP equation. We use the following convention.
For any $ n \in \NN $ and $ F \in C^1_b(\RR^{n}; \RR) $ and for any smooth
functions $ \varphi_{i} \in C^{\infty}_{c}(\RR), i = 1 ,\dots,n
$ write, for $ \boldsymbol{\varphi} = (\varphi_1, \ldots, \varphi_n) $ and any $ u \in C_{\mathrm{loc}}(\R\color{black}) $:
\begin{align*}
F_{\boldsymbol{\varphi}} (u) = F\Big(  \{ \langle u, \varphi_{i}  \rangle \}_{i = 1}^{n} \Big) \;, \qquad \partial_{i} F_{\boldsymbol{\varphi}} (u) =
\partial_{i}F\Big(  \{ \langle u,\varphi_{i}  \rangle \}_{i = 1}^{n} \Big)
\;.
\end{align*}
We also recall that for $ \mS_{\delta} $ as in Definition~\ref{def:R-delta} we
let $ \PP^{\delta} $ be the (random) probability distribution
\begin{align*}
\PP^{\delta} (\mA) = \PP ( \mA \vert \mS_{\delta})\;,
\end{align*}
and we let $ \mF_{t}^{\delta} $ be the filtration generated by
$$
\mF_{t}^{\delta} = \sigma (\mS_{\delta}^{t})\;, \qquad \mS_{\delta}^{t}
\overset{\mathrm{def}}{=}  \mS^{t} \cup \mS_{\delta}= ( \mS \cap (
[0, t] \times (0, 1]) )\cup \mS_{\delta}\;.$$

\begin{definition}[Conditional martingale solution]\label{def:martingale-solutions}
Fix any $ \mf{R} \in \mM $ as well as $ \delta \in (0, 1], \alpha \in (0,
1) $ and $ u_{0} \in C^{\alpha}_b $.
Let $ u $ be a stochastic
process over a probability space $ (\Omega, \mF, \PP) $ taking values in 
$ \DD \big([0, \infty); C^{\alpha}_b(\RR; [0, 1]) \big) $.
Let $ \mS_{\delta} = \{ (t_{j}, y_{j}) \}_{j \in \NN}$ be a Poisson point process as in
Definition~\ref{def:R-delta}, defined on the same probability space.

We say that $ u $ is a martingale solution to Equation~\eqref{eqn:fkpp} on $
[0, \infty) $ with initial condition $ u_{0}$, \emph{conditional on $
\mS_{\delta} $}, if $ u (0, \cdot) =
u_{0}(\cdot) $ and the following conditions are satisfied for any $n \in \NN $, $F\in C^1_b(\RR^n;\RR)$ and $ \varphi_{i} \in C^{\infty}_{c}(\RR)$ for $i = 1 ,\dots,n
$\color{black}:
\begin{enumerate}
\item For all $j \in \NN$ the process
\begin{align*}
M^{F}_{t} : = F_{\bp}(u_t) - F_{\bp}(u_{t_j}) - \int_{t_j}^t \mL^\delta (F_{\bp}) (u_s) \ud s 
\end{align*}
is an $ \mF_{t}^{\delta} $--c\`adl\`ag centered martingale for $t$ in $ [t_{j},
t_{j+1}) $, with $\mL^\delta (F_{\bp})$ defined as:
\begin{align*}
\mL^\delta (F_{\bp}) (u) = \left( \sum_{i=1}^n\partial_{ i } F_{\bp} (u) \cdot \Big( \langle u, \frac{1}{2}
\Delta \varphi_{ i } \rangle + \langle \mf{r} u(1-u) , \varphi_{ i }  \rangle  \Big) \right) +  \mL^{\mf R_{\delta}^{-}}(  F_{\bp} ) (u) \;. \numberthis\label{eqn:drift}
\end{align*}
\item For all $j\in \NN$ the martingale $ M^{F}_{t} $ has predictable quadratic
variation, for $ t \in [t_{j}, t_{j+1})$
\begin{align*}
\langle M^{F} \rangle_{t} = 
 \int_{t_{j}}^{t}  \mL^{\delta} ( ( F_{\bp} )^{2} )(u_s) - 2  F_{\bp} (u_s)
\mL^{\delta}( F_{\bp})(u_{s}) \ud s \;. \numberthis\label{eqn:qv}
\end{align*}
\item And finally for all $j \in \NN$ we have $ u_{t_{j}} = u_{t_{j}-} + y_{j} u_{t_{j}-}(1 -
u_{t_{j}-}) $.
\end{enumerate}
\end{definition}

In this setting we find the following result.

\begin{proposition}\label{prop:existence} Fix any $ \mf{R} \in \mM $, as well as  $
 \alpha \in (0, 1) $. Let $ (\Omega, \mF, \PP) $
be a probability space supporting a Poisson point process $ \mS = \mS_{0\color{black}}$ as in
Definition~\ref{def:R-delta}. For every $ u_{0} \in C^{\alpha}_b(\RR; [0, 1]) $ there  exists a unique solution
\[ u \colon \Omega \to \DD([0, \infty) ; C^{\alpha}_b(\RR; [0, 1])) \cap  \DD((0, \infty) ; C^{\beta}_b(\RR; [0, 1])) \]
to \eqref{eqn:fkpp} in the sense of Theorem~\ref{thm:existence-uniqueness}, for arbitrary $\beta  \geqslant 0$.
Moreover such $ u $ satisfies for any $ \delta \in (0, 1] $ the conditional
martingale problem of Definition~\ref{def:martingale-solutions}.
\end{proposition}

\begin{proof}

To construct solutions our approach \color{black} is to build an approximating sequence $ \{u^{\ve} \}_{\ve \in
(0, 1)} $
associated to Poisson point processes that has only finitely many jumps. For
this reason, for $\ve \in (0,1)$ we
recall that the measure $\mathcal R_{\ve}^+$ is the restriction of the Poisson
random measure $\mathcal R$ to $ [0, \infty) \times(\ve, 1] $, as in from
Definition~\ref{def:R-delta} (we call $ \mR $ a Poisson random measure, but it
can only be integrated against functions that vanish sufficiently
quickly near $ y=0 $, cf.~\eqref{eqn:Poissonintegral} for $\delta=1$).

In this way the Poisson random measure \( \mR_{\ve}^+(\ud t, \ud y) \) has
intensity $ \frac{1}{y}\mf{R}^+_\ve(\d y) \otimes \d t$ 
with finite total mass and we recall the representation through the locally
finite Poisson point process $ \mS_{\ve} = \{ (t_{i},y_{i}) \}_{i \in \NN}  $
\begin{align*}
\mathcal R_{\ve}^+ = \sum_{i \in \NN} \delta_{(t_{i}, y_{i})}\;.
\end{align*}
Let now $ u^{\ve} \in \DD([0, \infty) ; C^{\alpha}_b) $ be the
solution to:
\begin{equation}\label{eqn:approximate-fkpp}
\begin{aligned}
 \d u_{t}^{\ve} \color{black} & =  \frac{1}{2} \Delta u_{t}^{\ve}  \d t \color{black} + \mf{r}
u_{t}^{\ve}(1-u_{t}^{\ve})  \d t \color{black} +  \int_{(\ve ,1]} y u_{t
-}^{\ve}(1-u_{t-}^{\ve}) \mathcal R_{\ve}^+( \d t, \d y \color{black}),
\end{aligned}
\end{equation}
with initial condition $u_0 $. Here solutions are defined by the following
constraints for any $ j \in \NN $
\begin{align*}
u_{t}^{\ve} & = P_{t - t_{j}} u_{t_{j}}^{\ve}+ \int_{t_j}^t \mf{r} P_{t-s} [u^\ve_s(1-u^\ve_s) ] \ud s \;, \qquad \forall t \in [t_{j}, t_{j+1}) \;,\\
 u_{t_{j}}^\eps \color{black}& =  u_{t_{j}-}^\eps \color{black}+ y_{j}  u_{t_{j}-}^\eps \color{black}(1 -
 u_{t_{j}-}^\eps \color{black}) \;,
\end{align*}
where $ P_{t} $ is the heat semigroup $$ P_{t} \varphi(x) = \int_{\RR}
\frac{1}{\sqrt{2 \pi t}} \varphi(y)\mathrm e^{- \frac{| x-y |^{2}}{2 t}}\ud y\;.$$

\textit{Step 1: Existence and uniqueness.} We
start by proving that the sequence $
\{u^{\ve} \}_{\ve \in (0, 1)} $ is Cauchy in $ \DD([0, T]; C_{b}(\RR; [0, 1])) $
for $ \ve \to 0 $ and any $ T > 0 $. Fix any two $
0 < \overline{\ve} < \ve < 1 $ and let $ v^{\overline \ve,\ve}  $ be the difference $
v^{\overline \ve,\ve}  =  u^{ \overline{\ve}}- u^{\ve}$. Now we observe that $v^{\overline \ve,\ve}$ is positive: $v^{\overline \ve,\ve}_t(x) \geqslant 0$ for all $t \geqslant 0, x \in \RR$. This follows because the solution to the $\ve$-discretised FKPP equation is order preserving, meaning that if $g_0^1 \geqslant g_0^2$, then the solution $g_t^i$ for $i\in \{ 1,2\}$ to 
$\ud g^i = \Delta g^i \ud t + \mf{r} g^i (1- g^i) \ud t + \int_{(0, 1]} y g^i(1- g^i) \mR_\ve (\ud t , \ud y) $ satisfies $g^1_t \geqslant g_t^2$. Then observe that $u^{\overline{\ve}}$ solves the same equation as $u^\ve$ apart from the jump times $t_j$ such that $y_j \in (\overline{\ve}, \ve]$. In particular, if one assumes that at such times $t_j$ one has $v^{\overline\ve,\ve}_{t_j-} \geqslant 0$, then it follows that $v^{\overline \ve,\ve}_{t_j} \geqslant 0$, because $u^{\overline{\ve}}$ is increasing at time $t_j$: we can then conclude that $v^{\overline \ve,\ve}  \geqslant 0$ until the next time $t_{j^\prime}$ such that $y_{j^\prime} \in (\overline{\ve}, \ve]$. By induction, since $v^{\overline{\ve}, \ve}_0 = 0$ we obtain as desired that $v^{\overline{\ve},\ve}_t\geqslant 0$ for all $t \geqslant 0$.
In addition, $ v^{\overline{\ve}, \ve} $ solves
\begin{align*}
\ud v^{\overline \ve,\ve}  =& \frac{1}{2} \Delta v^{\overline \ve,\ve} \ud t + \mf{r} v^{\overline \ve,\ve}(1-u^{\ve}-u^{\overline \ve}) \d t  + \int_{(0, 1]} y v^{\overline \ve,\ve}(1 - u^{\ve} - u^{ \overline{\ve}})
\mR_{\ve}^{+}(\ud t, \ud y) \\
& + \int_{(0,1]} y  u^{ \overline{\ve}}(1 - u^{ \overline{\ve}})
\mathds 1_{( \overline{\ve}, \ve]}(y) \mR_{ \overline{\ve}}^{+}(\ud t, \ud y) \\
\leqslant & \left( \frac{1}{2} \Delta v^{\overline \ve,\ve} + \mf{r} v^{\overline \ve,\ve}  \right)\ud t + \int_{(0, 1]}y v^{\overline \ve,\ve} 
\mR_{\ve}^{+}(\ud t, \ud y) + \int_{(0,1]} y  
\mathds 1_{( \overline{\ve}, \ve]}(y) \mR_{ \overline{\ve}}^{+} ( \ud t, \ud y)\;,
\end{align*}
where we used that the solution takes values in $ [0, 1] $. Using that $
v^{\overline \ve,\ve}  (0, \cdot) = 0$, we find via a maximum principle the upper bound:
\begin{align*}
\| v_{t}^{\overline \ve,\ve}  \|_{\infty} & \leqslant \int_{0}^{t} \int_{(0,1]} y \mathds 1_{( \overline{\ve} , \ve]}(y)
\exp\left\{ \mf{r}(t-s) + \int_{s}^{t} \int_{(0,1]} \log{(1 + \overline{y})}  \mR_{\ve}^{+} (
\ud r, \ud \overline{y}) \right\}
\mR_{ \overline{\ve}}^{+}( \ud s, \ud y \color{black}) \\
& \leqslant \int_{0}^{t} \int_{(0,1]} y \mathds 1_{(0 , \ve]}(y)
\exp\left\{ \mf{r} (t-s)+\int_{s}^{t} \int_{(0,1]} \log{(1 + \overline{y})}  \mR( \ud r, \ud
\overline{y}) \right\} \mR( \ud s, \ud y \color{black})\;,
\end{align*}
where the integrals are defined in the sense of Campbell's theorem (see~\eqref{eqn:Poissonintegral}, or \cite[Section 3]{K93} for further details). Now, the
right-hand side decreases to zero as $ \ve \to 0 $,
uniformly over $t \in [0, T]$:
\begin{align*}
    \lim_{\ve \to 0}\sup_{t \in [0, T]} \|v^{\overline{\ve},\ve}_t\|_{\infty} =0 \;.
\end{align*}
Since the supremum norm (in time) dominates the Skorohod distance, the claimed
convergence in $\DD([0, T]; C_b(\RR; [0,1]))$ follows.
We observe that the same argument delivers also uniqueness of solutions, if we
replace $ u^{ \overline{\ve}} $ with any solution $u $ to \eqref{eqn:fkpp}.

 \textit{Step 2: Regularity.} Now we focus on the $C^\alpha_b$ regularity of the solutions $u^\ve$. Since the only issue for the regularity comes from the
Poisson jumps let us assume without loss of generality that $ \mf{r} = 0$. The argument we present works verbatim to obtain the required regularity for the limiting solution $u$: We provide instead a bound that is uniform in $\ve$ instead, which will be useful to obtain the martingale problem in our last step. The bound reads as follows: 
\begin{align*}
\sup_{\ve \in (0, 1)} \sup_{t \in [0, T]} \| u^{\ve}_{t} \|_{C^{\alpha}_b} < C <
\infty\;,
\end{align*}
for some \color{black} random constant $ C > 0 $.
To prove this statement we can write $ u^{\ve}$ in its mild formulation, as a
convolution with the heat semigroup:
  \begin{align*}
    u_{t}^{\ve} = P_{t} u_{0} + 
    \int_{0}^{t} \int_{(0, 1]} y P_{t - s} u_{s-}^{\ve}
(1 - u_{s-}^{\ve}) \mathcal R_{\ve}^{+}(  \ud s, \ud y \color{black}) \;.
  \end{align*}
Then we can use classical Schauder estimates, which captures the regularisation
of the Laplacian
\begin{align*}
\| P_{t} \varphi \|_{C^{\alpha}_b} \leqslant C(\alpha, \gamma) t^{- \frac{\gamma}{2}} \| \varphi \|_{\infty} \;,
\end{align*}
for some $C(\alpha, \gamma) > 0$ and all $\alpha, \gamma$ such that $\alpha, \gamma \in (0, \infty)
\setminus \NN$, $\gamma > \alpha$ and  $\alpha + \gamma \in (0, \infty) \setminus \NN$ (this follows for example from \cite[Theorem 2.24]{BAH} by embedding $L^\infty$ in a Besov space of negative regularity).

Then we find for any $2>\gamma > \alpha $
\begin{equation}\label{eqn:regularity-estimate}
\begin{aligned}
\| u_{t}^{\ve} \|_{C^{\alpha}_b}&  \lesssim \| u_{0} \|_{C^{\alpha}_b} +
\int_{0}^{t} \int_{(0,1]} y \| u^{\ve}_{ s-\color{black}} (1 - u^{\ve}_{ s-\color{black}}) \|_{\infty} (t-s)^{-
\frac{\gamma}{2}}\color{black} \mR_{\ve}^{+} ( \ud s, \ud y \color{black}) \\
& \lesssim \| u_{0} \|_{C^{\alpha}_b} + \int_{0}^{t}  \int_{(0,1]} (t-s)^{-
\frac{\gamma}{2}}\color{black}  y
\mR(  \ud s, \ud y \color{black}) \;,
\end{aligned}
\end{equation}
where the latter integral is again defined by~\eqref{eqn:Poissonintegral}, since for $\gamma \in (0, 2)$
\begin{equation*}
    \int_0^t\int_{(0,1]} (t-s)^{-
\frac{\gamma}{2}} y \  \frac{1}{y} \mf{R} (\ud y) \ud s < \infty \;.
\end{equation*}
Hence our upper
bound is proven. We observe that we can additionally deduce the
following moment bound:
\begin{equation*}
 \EE \Big[ \sup_{\ve \in (0, 1)} \sup_{0 \leqslant t \leqslant T}  \|
u^{\ve}_{t} \|_{C^{\alpha}_b}  \Big] < \infty \;.
\end{equation*}
The bound for arbitrary $\beta \geqslant 0$ follows similarly, allowing a blow-up at $t= 0$. 

\newcommand{\bphi}{\boldsymbol{\varphi}}
\textit{Step 3: Martingale problem.}  Since the proof does not vary, we restrict
to establishing this property for $ \delta =1 $.
 Hence, consider $ F $ as in
Definition~\ref{def:martingale-solutions}. It is straightforward to establish
the martingale problem for $ u^{\ve} $ (where we are in presence  of \color{black} locally
finite jumps). Namely, we have that
\begin{align*}
M^{F, \ve}_{t} : = &  F_{\boldsymbol \varphi} (u_t^{\ve})-F_{\boldsymbol \varphi}(u_0^{\ve}) - \int_0^t \mL^{1, \ve} (F_{\bphi}) (u_s^{\ve}) \ud s
\end{align*}
is a c\`adl\`ag martingale, with 
\begin{align*}
  (\mL^{1, \ve} F_{\bphi} ) (u) = \left( \sum_{i=1}^n\partial_{i} F_{\bphi} (u) \cdot \Big( \langle u, \frac{1}{2}
\Delta \varphi_{ i } \rangle  + \langle \mf{r} u(1-u) , \varphi_{ i }  \rangle  \Big) \right) 
 + \mL^{\mf R_{\ve}^{+}}( F_{\bphi} ) (u) \ud s \;,
\end{align*}
in analogy to \eqref{eqn:drift} (the notation $\mL^{1, \ve}$ is used to avoid confusion with $\mL^\delta$, since we are in the case $\delta=1$).
Moreover, $ M^{F, \ve}_{t}$ has predictable quadratic variation
\begin{align*}
\langle M^{F,\ve} \rangle_{t} = 
 \sum_{i  = 1}^{n}\int_{t_{j}}^{t}  \mL^{1,\ve} ( ( F_{\bp} )^{2} )(u_s^\ve) - 2  F_{\bp} (u_s^\ve)
\mL^{1,\ve}( F_{\bp})(u_{s}^\ve) \ud s \;.
\end{align*}
At this point we would like to establish the continuity
\begin{equation}\label{e:aim-4}
    \mL^{\mf{R^+_\ve}}( F_{\bphi} ) (u_{s}^{\ve}) \to \mL^{\mf R} ( F_{\bphi} ) (u_{s})\;, \qquad \mL^{\mf R_{\ve}^{+}}( (F_{\bphi})^2 ) (u_{s}^{\ve}) \to \mL^{\mf R} ( (F_{\bphi})^2 ) (u_{s})  \;.
\end{equation}
Since from the established convergence of $u^\ve$ we know that $F_{\bphi}(u^\ve_s) \to F_{\bphi}(u_s)$ almost surely, \eqref{e:aim-4} would guarantee that $M^{F, \ve}_{\cdot}$ converges to $M^F_{\cdot}$ almost surely, and similarly the quadratic variation at level $\ve$ would converge to the desired limiting quadratic variation almost surely. Now, to establish \eqref{e:aim-4} we can compute for any pair $u,v \in C_b(\RR;[0,1])$:
\begin{align*}
    \left\vert \mL^{\mf R_{\ve}^{+}}(F_{\bp} ) (u) - \mL^{\mf R} ( F_{\bphi} ) (v) \right\vert \leqslant & \left\vert \mL^{\mf R_{\ve}^{+}}(F_{\bp} ) (u) - \mL^{\mf R} ( F_{\bp} ) (u) \right\vert + \left\vert \mL^{\mf R} ( F_{\bp} ) (u) - \mL^{\mf R} ( F_{\bp} ) (v) \right\vert  \;.
\end{align*}
Now, for the first term we have via Remark~\ref{rem:bd-gen}
\begin{align*}
    \left\vert \mL^{\mf R_{\ve}^{+}}(F_{\bp} ) (u) - \mL^{\mf R} ( F_{\bp} ) (u) \right\vert & = \left\vert \int_{(0,\ve]} \left\{ F_{\bp}(u + y u(1-u)) - F_{\bp} (u) \right\} \frac{1}{y} \mf{R} (\ud y) \right\vert \\
    & \lesssim_{\|F \|_{C^1_b,},\|\varphi_i\|_{L^1}} \mf{R}((0,\ve]) \to 0 \;.
\end{align*}
As for the second one we have
\begin{equation}\label{e:second-term}
    \left\vert \mL^{\mf R} ( F_{\bp} ) (u) - \mL^{\mf R} ( F_{\bp} ) (v) \right\vert \leqslant \int_{(0,1]} f(u, v, y) \frac{1}{y} \mf{R}(\ud y)\;,
\end{equation}
where for $u,v \in C_b(\RR;[0,1])$ we have
\begin{align*}
    f(u, v, y) & = |\left\{ F_{\bp}(u + y u(1-u)) - F_{\bp} (u) \right\} - \left\{ F_{\bp}(v + y v(1-v)) - F_{\bp} (v) \right\}| \\
    & \leqslant C\big( \|F\|_{C^1_b}, \sum_{i=1}^n \|\varphi_i\|_{L^1} \big) \left\{ y \wedge  \|u - v\|_{\infty} \right\}\;.
\end{align*}
Then, since $\|u-v\|_{\infty} \leqslant 1$ we can split up the integral in \eqref{e:second-term} in the two intervals $[0, \sqrt{\|u-v\|_{\infty}}]$ and $[\sqrt{\|u-v\|_{\infty}}, 1]$. We then obtain a bound of the form
\begin{align*}
    \left\vert \mL^{\mf R} ( F_{\bp} ) (u) - \mL^{\mf R} ( F_{\bp} ) (v) \right\vert \lesssim_{\|F \|_{C^1_b}, \|\varphi_i\|_{L^1}} \mf{R}((0, \sqrt{\|u - v\|_{\infty}}]) + \sqrt{\|u - v \|_{\infty}} \mf{R}((0, 1])\;,
\end{align*} which converges to zero if $u \to v$ in $C_b(\RR; [0, 1])$. Here for the first term we have used Remark~\ref{rem:bd-gen}, and for the second term we have estimated
\begin{align*}
    \int_{\sqrt{\|u-v\|_{\infty}}}^1 \|u-v\|_{\infty} \frac{1}{y} \mf{R}(\ud y) \leqslant \sqrt{\|u - v\|_{\infty}} \mf{R}((0, 1]) \;.
\end{align*}
The second convergence in \eqref{e:aim-4} follows similarly. Finally, from \eqref{e:aim-4} we then find that $ M^{F,
\ve}_{\cdot} $ converges almost surely to $ M^{F}_{\cdot} $ in $ \DD([0, T]; \RR)
$, for $ M^{F} $ as in Definition~\ref{def:martingale-solutions}. To conclude
that $ M^{F} $ is a martingale, we can bound following the same arguments we just outlined:
\begin{align*}
\sup_{\ve \in (0,1)} \EE | M^{F, \ve}_{T} |^{2} < \infty \;.
\end{align*}
Hence by uniform integrability the limiting point $ M^{F} $ is also a
martingale, and a similar argument as above shows that it has the required predictable quadratic
variation.
\end{proof}

We conclude with an extension of the previous results to a case in which $F$ is not of the prescribed form $F = F_{\bp}$. To state this assertion, for $ \mathbf{x} \in \mP $ of the form $ \mathbf{x} = (x_{1}, \dots ,
x_{n}) \in \RR^{n} $, we write $$  \mathbf{x} \dagger x_{i} = (x_{1}, \dots, x_{i-1}, x_{i+1}, \dots x_{n}) \in
\RR^{n-1}.$$
In addition for $\mathbf{z} \in \mP$, we write $\mathbf{z} \subseteq \mathbf{x}$ if there exists $i_1 < \ldots < i_m$, with $m \leqslant n$, such that
\begin{align*}
    \mathbf{z} = (x_{i_1}, \ldots, x_{i_m})\; \qquad  \text{ and } \qquad \mathbf{x} \dagger \mathbf{z} = (x_i \, \colon \, i \neq i_j \;, \forall j \in \{1, \ldots, m\})\;.
\end{align*}

\begin{lemma}\label{lemma-dualityfunction}
Assume that $u$ is a martingale solution to Equation~\eqref{eqn:fkpp} with initial condition $u_0$, in the sense of Definition~\ref{def:martingale-solutions}, and $u(0,\cdot)=u_0(\cdot)$ where $u_0\in C^\alpha_b$ for $\alpha \in (0,1)$. Then, for any fixed $\delta \in (0,1], \mathbf{x} \in \mP $ and any jump time $t_j > 0$ of $ \mS_{\delta} $ we have that 
\begin{align*}
    (1- u_{t})^{\mathbf{x}} - & \int_{t_{j}}^{t}\sum_{x \in \mathbf{x}}
(1-u_{s})^{\mathbf{x} \dagger x}  \left\{ \frac{1}{2} \Delta
(1- u_{s})(x) + \mathfrak{r} \, u_s(x) (u_s(x)-1)\right\}\ud s\\
- & \int_{t_j}^t\sum_{\mathbf{z} \subseteq \mathbf{x}} \int_{(0,\delta]} y^{\ell(\mathbf{z})} (1-y)^{\ell(\mathbf{x} ) - \ell(\mathbf{z})} \left\{ (1-u_s)^{\mathbf{x} \sqcup \mathbf{z}} - (1-u_s)^{\mathbf{x}} \frac{1}{y} \mf{R}(\ud y) \right\} \ud s \; \numberthis\label{eqn:martingale1-u}
\end{align*}
is a square integrable martingale for $ t \in [t_{j}, t_{j+1}) $, with respect to the filtration $\mF^\delta$. 
\end{lemma}
Note that the quantity in~\eqref{eqn:martingale1-u} is well-defined, since $ u $ is smooth for strictly
positive times: $ u_{s} \in C^{\infty}_b(\RR) $ for any $ s > 0 $ by Proposition~\ref{prop:existence}.
\begin{proof}
Fix any non-negative smooth function $\varphi \in C_c^\infty$ such that $\int_\RR \varphi(x)\d x =1$. Then for all $\zeta>0$ and $y \in \RR$ define
\[ \varphi_{x,\zeta}(y)=\frac{1}{\zeta} \varphi\Big( \frac{1}{\zeta} (y-x) \Big). \]
Now, consider
\[ F^{(\zeta)}(v) = \prod_{x\in\mathbf x} \langle 1-v, \varphi_{x,\zeta} \rangle, \qquad v \in C_{\mathrm{loc}}(\RR). \]
Since $u$ is a martingale solution to Equation~\eqref{eqn:fkpp} with initial condition $u_0$, conditional on $S_\delta$, and since $v 
\mapsto F^{(\zeta)} (v) $ is smooth and bounded over $v\in [0,1]$, we can apply Definition~\ref{def:martingale-solutions} to obtain that
\begin{equation}\label{e:marting-approx}
\begin{aligned}
    F^{(\zeta)} (u_t) & - \int_{t_j}^t  \sum_{x \in \mathbf{x} }\Big( \prod_{y\in\mathbf x \dagger x} \langle 1-u_s, \varphi_{y,\zeta} \rangle \Big) \cdot \Big\{  \langle 1-u_{s}, \frac{1}{2}
\Delta \varphi_{x, \zeta} \rangle  + \langle \mf{r} \, u_{s}(u_{s}-1) , \varphi_{ x, \zeta}  \rangle  \ud s\Big\} \\
& - \int_{t_j}^t \int_{(0, \delta]}\Big( F^{(\zeta)}(u_s +yu_s (1-u_s)) - F^{(\zeta)}(u_s) \Big) \frac{1}{y} \mf{R}(\ud y) \ud s
\end{aligned}
\end{equation}
is a martingale on $[t_j, t_{j+1})$.
In addition, for $F^{(0)} (u) = (1-u)^{\mathbf{x}}$, there exists a constant $C>0$ such that
\begin{align*}
    \sup_{ \|u\|_{\infty}, \|w\|_{\infty} 
    \leqslant 1}|F^{(0)}(u+yw) - F^{(0)}(u)| \leqslant C y \;.
\end{align*}
Therefore, since $F^{(\zeta)} (u + y u(1-u)) = F^{(0)}(w +y r)$ with
\begin{align*}
    w(x) = \langle u, \varphi_{x, \zeta} \rangle \;, \qquad r(x) = \langle u(1-u), \varphi_{x, \zeta} \rangle \;,
\end{align*}
so that $\|w\|_{\infty}, \|r\|_{\infty} \leqslant 1$, we also have the following bound which is uniform over $\zeta$ (for the same constant $C>0$ as above):
\begin{equation}\label{e:unif-m}
      \sup_{ \zeta \in (0, 1),  \|u\|_{\infty}, \|w\|_{\infty} 
    \leqslant 1}|F^{(\zeta)}(u+yw) - F^{(\zeta)}(u)| \leqslant C y \;.
\end{equation}
Now, for $\zeta \to 0$ we have that $F^{(\zeta)}(u) \to F^{(0)} (u)$ point-wise. The uniform bound \eqref{e:unif-m} guarantees that we can pass to the limit $\zeta \to 0$ under the integral over $y$ in \eqref{e:marting-approx} and moreover via Definition \ref{def:martingale-solutions} we see that limit is still a martingale as the quadratic variation stays uniformly bounded, again by \eqref{e:unif-m}. We have therefore concluded that
\begin{align*}
    F^{(0)} (u_t) & - \int_{t_j}^t  \sum_{x \in \mathbf{x} } ( 1-u_s)^{\mathbf{x} \dagger x} \cdot \left\{ \frac{1}{2} \Delta
(1- u_{s})(x) + \mathfrak{r} \, u_s(x) (u_s(x)-1)\right\}\ud s \\
& - \int_{t_j}^t \int_{(0, \delta]}\Big( F^{(0)}(u_s +yu_s (1-u_s)) - F^{(0)}(u_s) \Big) \frac{1}{y} \mf{R}(\ud y) \ud s
\end{align*}
is a martingale.
Finally, we must obtain that the last term coincides with the term in the statement of the lemma. To see this, we  compute 
\begin{align*}
    (1 - u - yu(1-u))^{\mathbf{x}}= (1- u)^{\mathbf{x}}(1-yu)^{\mathbf{x}} = (1-u)^{\mathbf{x}}(1-y +y (1-u))^\mathbf{x} \;.
\end{align*}
Then we use the binomial formula:
\begin{align*}
    (u+v)^{\mathbf{x}}=\sum_{\mathbf{z} \subseteq \mathbf{x}} u^\mathbf{z}v^{\mathbf{x \dagger z}} \;.
\end{align*}
In particular, we obtain that
\begin{align*}
    (1-y +y (1-u))^\mathbf{x} = \sum_{\mathbf{z} \subseteq \mathbf{x}}  y^{\ell(\mathbf{z})} (1-y)^{\ell(\mathbf{x} \dagger \mathbf{z})} (1-u)^\mathbf{z} \;,
\end{align*}
so we can finally rewrite:
\begin{align*}
    \int_{(0,\delta]}(1 - u & - y(u(1-u)))^{\mathbf{x}}  -(1-u)^{\mathbf{x}} \frac{1}{y} \mf{R}(\ud y) \\
    & = \sum_{\mathbf{z} \subseteq \mathbf{x}} \int_{(0,\delta]} y^{\ell(\mathbf{z})} (1-y)^{\ell(\mathbf{x} \dagger \mathbf{z})} \left\{ (1-u)^{\mathbf{x}\sqcup \mathbf{z}} - (1-u)^{\mathbf{x}} \right\} \frac{1}{y} \mf{R}(\ud y) \;,
\end{align*}
from which our claim follows.
\end{proof}

\subsection{Duality}\label{sec:duality}

As we have already discussed, we will consider the solution
to \eqref{eqn:fkpp} conditional on the large jumps $ \mR^{+}_{\delta} $. 
In particular, the solution $ u $ to \eqref{eqn:fkpp} can be
formally rewritten as
\begin{equation}\label{eqn:fkpp-cndtnl}
\begin{aligned}
\ud u_{t} = & \frac{1}{2} \Delta u_{t} \ud t + \mf{r} \, u_{t}(1 - u_{t}) \ud t  \\
&  + \int_{(0, \delta]} y u_{t-}(1 - u_{t-}) \mR_{\delta}^{-}( \ud t, \ud y \color{black}) +
\int_{(\delta, 1]} y  u_{t-} (1 - u_{t-}) \mR^{+}_{\delta} ( \ud t, \ud y \color{black}) \;.
\end{aligned}
\end{equation}
Here the integral against $ \mR^{-}_{\delta} $ should be interpreted in the
sense of~\eqref{eqn:Poissonintegral}.  Then, let $ \EE^{\delta} $ indicate expectation conditional on $ \mS_{\delta} $
as in~\eqref{eqn:def-skeleton}, or equivalently conditional on $
\mR^{+}_{\delta} $ as in Definition~\ref{def:R-delta}:
\begin{align*}
\EE^{\delta}[f] = \EE [f \vert \mS_{\delta}] \;.
\end{align*}
For \( u \in C_{\mathrm{loc}} (\RR)\) let us recall the notation
\begin{align*}
u^{\mathbf{C}_{t}} = \prod_{i =1}^{\ell(\mathbf{C}_{t})} u(C^{(i)}_{t}) \;.
\end{align*}
We find the following duality relation.
\begin{proposition}\label{prop:cond-dual}
Fix $ \mf{R} \in \mM $ and, for any $ \delta \in (0, 1] $, let $
\mf{R}_{\delta} = (\mf{R}_{\delta}^{-}, \mf{R}_{\delta}^{+}) $ be as in
Definition~\ref{def:R-delta}. Then, let $ u $ be a martingale solution to
\eqref{eqn:fkpp-cndtnl} conditioned on $\mS^\delta$, associated to $ \mf{R} $ in the sense of
Definition~\ref{def:martingale-solutions}. Furthermore, for any $ \mathbf{x} \in \mP$, let $
(\mathbf{C}_{t})_{t \geqslant 0}$ be an $ \mf{R}_{\delta}$--CBBM started in $ \mathbf{x} $ as in Definition
\ref{def:rs-cbbm}. Then for any $ t > 0 $
\begin{align}\label{eqn:duality}
\EE^{\delta} \left[ (1 - u_{t})^{\mathbf{x}} \right]= \EE^{\delta} \big[ (1 -
u_{0})^{\mathbf{C}_{t}} \big] \;.
\end{align}
\end{proposition}

\begin{proof} 
Since whether Equation~\eqref{eqn:duality} holds only depends on the mar\-ginal laws of the couple $ ((\mathbf C_t)_{t\geqslant 0}, (u_t)_{t\geqslant 0})$ under the (random) probability measure $\PP^\delta$, we can without loss of generality assume that the two processes are independent conditional on $\mathcal S_\delta$.

Our aim is then to prove an even stronger statement, namely
that for any $ t > 0 $ the process
\begin{equation}\label{eqn:aim-duality}
[0, t] \ni s \mapsto \EE^{\delta} (1 -u_{t- s})^{\mathbf{C}_{s}} \ \text{ is constant. } 
\end{equation}
From the definition of $\mf{R}_{\delta}$-CBBM we find that
\begin{align*}
 (1 - u)^{\mathbf{C}_{t}} - & \int_{t_{j}}^{t}\sum_{x \in \mathbf{C_s}}
(1-u)^{\mathbf{C}_s \dagger x}  \left\{ \frac{1}{2} \Delta
(1- u)(x) + \mathfrak{r} \, u(x) (u(x)-1)\right\}\ud s\\
- & \int_{t_j}^t\sum_{\mathbf{z} \subseteq \mathbf{C}_{s-}} \int_{(0,\delta]} y^{\ell(\mathbf{z})} (1-y)^{\ell(\mathbf{C}_{s-} ) - \ell(\mathbf{z})} \left\{ (1-u)^{\mathbf{C}_{s-} \sqcup \mathbf{z}} - (1-u)^{\mathbf{C}_{s-}} \frac{1}{y} \mf{R}(\ud y) \right\} \ud s \;,
\end{align*}
is a square integrable martingale on $[t_j, t_{j+1})$, for any fixed $u\in C^2_b$.
In addition, by Lemma~\ref{lemma-dualityfunction} we have that also
\begin{align*}
    (1- u_{t})^{\mathbf{x}} - & \int_{t_{j}}^{t}\sum_{x \in \mathbf{x}}
(1-u_{s})^{\mathbf{x} \dagger x}  \left\{ \frac{1}{2} \Delta
(1- u_{s})(x) + \mathfrak{r} \, u_s(x) (u_s(x)-1)\right\}\ud s\\
- & \int_{t_j}^t\sum_{\mathbf{z} \subseteq \mathbf{x}} \int_{(0,\delta]} y^{\ell(\mathbf{z})} (1-y)^{|\mathbf{x} | - \ell(\mathbf{z})} \left\{ (1-u_s)^{\mathbf{x} \sqcup \mathbf{z}} - (1-u_s)^{\mathbf{x}} \frac{1}{y} \mf{R}(\ud y) \right\} \ud s 
\end{align*}
is a square integrable martingale on $[t_j, t_{j+1})$.
In particular, since the two
drifts match each other and the processes are assumed to be independent, upon taking expectations we find that for $ t \in [t_{j}, t_{j+1}) $ 
\begin{align*}
s \mapsto \EE^{\delta} (1 -u_{t- s})^{\mathbf{C}_{s}} \text{ is constant on }
[t_{j}, t]\;. 
\end{align*}

Now, 
at time $ t_{j} $, we find for $ z_{t} = 1 - u_{t} $
\begin{align*}
  (1 - u_{t_{j}})^{\mathbf{x}}& =  (1 - u_{t_{j -} } - y_{j} u_{t_{j}-}(1 -
  u_{t_{j}-}))^{\mathbf{x}} \\
  & =  ( z_{t_{j}-} - y_{j} (1 - z_{t_{j}-})z_{t_{j}-})^{\mathbf{x}} \\
  & = ( (1-y_{j}) z_{t_{j -} } + y_{j} z_{t_{j}-}^{2} )^{\mathbf{x}} \;.
\end{align*}
Hence in particular
\begin{align*}
 (1 - u_{t_{j}})^{\mathbf{x}} - (1 - u_{t_{j}-})^{\mathbf{x}}& = \bigg\{ \sum_{\mathbf{z}
\subseteq  \mathbf{x}} y_{j}^{\ell(\mathbf{z})}(1 -
y_{j})^{\ell(\mathbf{x} \dagger \mathbf{z})} z_{t_{j}-}^{\mathbf{x}\dagger
\mathbf{z}} ( z_{t_{j}-}^{2})^{\mathbf{z}} \bigg\} - z_{t_{j}-}^{\mathbf{x}} \\
& = \sum_{\mathbf{z} \subseteq  \mathbf{x}} y_{j}^{\ell(\mathbf{z})}(1 -
y_{j})^{\ell(\mathbf{x} \dagger \mathbf{z})} \bigg\{ z_{t_{j}-}^{\mathbf{x}\dagger \mathbf{z}} (
z_{t_{j}-}^{2})^{\mathbf{z}} - z_{t_{j} -}^{\mathbf{x}} \bigg\} \;,
\end{align*}
where we used that 
\begin{align*}
\sum_{\mathbf{z} \subseteq  \mathbf{x}} y_{j}^{\ell(\mathbf{z})}(1 -
y_{j})^{\ell(\mathbf{x} \dagger \mathbf{z})} = 1\;.
\end{align*}
This corresponds again to the branching mechanism of $ \mathbf{C}_{t} $, so
that we can deduce that $ \EE^{\delta} (1 - u_{t-s})^{\mathbf{C}_{s}} $ is
constant for $ s \in (0, t) $. Taking the limit $ s \downarrow 0 $ and $ s
\uparrow t $ delivers the result on the closed interval $ [0, t] $.
\end{proof}

\section{Wave speed}\label{sec:wavespeed}

\subsection{Conditional dual}

Apart from the results of Section~\ref{sec:existenceduality}, the last ingredient we will need in the study of the wave speed is the following
almost sure asymptotic property for $ \mS_{\delta}$.
\begin{lemma}\label{lem:asympt-skltn}
For any $ \delta \in (0, 1) $ and $ \ve \in [0, \infty) $ let $ \mS_{\delta} $ be defined as in
\eqref{eqn:def-skeleton} and associated to $ \mf{R}_{\delta} =
(\mf{R}^{-}_{\delta}, \mf{R}^{+}_{\delta}) $ as in Definition~\ref{def:R-delta}. Then for
\begin{align*}
  \mf{d}_{\delta, \ve}(t) = \sum_{j  \, \colon \, t_{j} \leqslant t} \log(1 + y_{j}
+ \ve)\;,
  \quad \mf{d}_{\delta, \ve} = \int_{(\delta,1]} \log{ (1+y+ \ve) }\frac{1}{y}
\mf{R} ( \ud y)\;,
\end{align*}
it holds that almost surely that
\begin{align*}
  \lim_{t \to \infty} \frac{1}{t} \mf{d}_{\delta, \ve}(t) = \mf{d}_{\delta, \ve} \; .
\end{align*}
\end{lemma}
For simplicity we will write $ \mf{d}_{\delta}(t)$ for $
\mf{d}_{\delta, 0}(t) $ in the case $ \ve =0 $ (and similarly for $
\mf{d}_{\delta} $).

\begin{proof}
We have
\[ \frac{1}{t} \mf{d}_{\delta, \ve}(t) = \Big( \frac{1}{\# \{ j  \, \colon \, t_{j} \leqslant t \} }\sum_{j  \, \colon \, t_{j} \leqslant t} \log(1 + y_{j}
+ \ve) \Big) \frac{\# \{ j  \, \colon \, t_{j} \leqslant t \} }{t} \;.\]
Now, the first factor on the right-hand side converges a.s.\ to $\frac{\int_{(\delta,1]} \frac{\log(1+y+\eps)}{y} \mf R(\d y)}{\int_{(\delta,1]} \frac{1}{y} \mf R(\d y)}$ due to the strong law of large numbers, while the second factor converges a.s.\ to $\int_{(\delta,1]} \frac{1}{y} \mf R(\d y)$ due to the Poisson law of large numbers \cite[Section 4.2]{K93}.
\end{proof}
\color{black}

\subsection{Upper bound on the wave speed}

We start by establishing an upper bound on the quenched (w.r.t.\ jumps larger than $\delta$) growth rate of the dual
process.
\begin{proposition}\label{prop:globalsurvival}Fix $ \mf{R} \in \mM $ and, for any $ \delta \in (0, 1] $, let $
\mf{R}_{\delta} = (\mf{R}_{\delta}^{-}, \mf{R}_{\delta}^{+}) $ be as in
Definition~\ref{def:R-delta}. Let $ \mathbf{C}_{t} $ be an $
\mf{R}_{\delta} $--CBBM started in $ \mathbf{x}\in
\RR^{1} $. Then for $ I_{t}^{\delta} = \ell(\mathbf{C}_{t}) $ we have almost surely
\[ \limsup_{t\to\infty} \frac{1}{t} \log I_t^{\delta} \leqslant
 \mf{R}_{\delta}^{-}([0, \delta]) + \int_{(\delta, 1]} \log{ (1+y) }
 \frac{1}{y} \mf{R}^{+}_{\delta} ( \ud y) =: \mf{c}_{\delta} \;.
\numberthis\label{eqn:logspeed} \]
\end{proposition}
Note that since $ \log{(1 + y)} \frac{1}{y} \leqslant 1$ we always have
\begin{align*}
\mathfrak{R}(\{0\}) + \int_{(0,1]} \log{(1 + y)} \frac{1}{y} \mf{R}(\ud y) \leqslant \mf{c}_{\delta}
\;,
\end{align*}
which reflects the fact that in the following we obtain an \emph{upper} bound on the wave speed. On the other hand, as $\delta \to 0$ we obtain 
\begin{equation}\label{e:mfc}
    \mf{c}_\delta \to \mathfrak{R}(\{0\})+\int_{(0,1]} \log{(1 + y)} y^{-1}\mf{R}(\ud y) = \mf{c} = \frac{\mathfrak{s}^2}{2}\;,
\end{equation}
where $\mathfrak{s}$ is the wave speed defined by \eqref{eqn:def-speed}.
\begin{remark}\label{rem:lower-bound}
Although we only prove an upper bound on the growth rate of $
I^{\delta}_{t} $, the arguments we present for the lower bound of the wave
speed allow to also prove a lower bound to the growth of $ I^{\delta} $ by comparing $ I^{\delta} $ to the
a process where we do not have any jumps with impact $ y \in (0, \delta] $
(and everything else being left unvaried). Alternatively, one can also use the same
comparison, but in combination with the ``channelling'' argument that we will
use in Section~\ref{sec:invexp}. Overall, one would thus obtain that almost surely
(independently of $ \delta $!)
\begin{align*}
\lim_{t \to \infty} \frac{1}{t} \log{I^{\delta}_{t}} = \mf{c}\;,
\end{align*}
with $\mf{c}$ as in \eqref{e:mfc}.
\end{remark}
The proof of Proposition~\ref{prop:globalsurvival} will be carried out in
Section~\ref{sec:invexp}.
Using the previous result, we obtain an upper bound on the wave speed via a
quenched version of the so-called many-to-one lemma, cf.\ \cite[Section 3.6]{Berestycki}.

\begin{proposition}\label{prop:manytoone}
Fix $ \mf{R} \in \mM $ and, for any $ \delta \in (0, 1] $, let $
\mf{R}_{\delta} = (\mf{R}_{\delta}^{-}, \mf{R}_{\delta}^{+}) $ be as in
Definition~\ref{def:R-delta}.
Let $ \mathbf{C}_{t} $ be an $ \mf{R}_{\delta}$--CBBM
started in $ \mathbf{x} = 0 \in \RR^{1} $. Then for any $ x_{0} \in \RR $ and $ \mathbf{S}_{t} = \max
\mathbf{C}_{t} $ we have, for any $ \lambda > \sqrt{2
\mf{c}_{\delta}} $, $ \PP $--almost surely
\[ \lim_{t \to \infty} \PP^{\delta}( \mathbf{S}_{t} > \lambda t + x_{0}) = 0 \;. \] 
\end{proposition}
\begin{proof}
Without loss of generality we restrict to the case $ x_{0}=0 $.
Let us also write $\mathbf C_t=(C_t^{(i)})_{i \in [I_t^{\delta}]}$ for $t>0$
and $ I^{\delta}_{t} = \ell( \mathbf{C}_{t})$, where we assume
that the ordering of particles is exchangeable (which can e.g.\ be achieved via
reshuffling the indices at the time of any reproduction event uniformly in a
right-continuous manner).
Then, at any given time $t \geq 0$, conditional on $
I^{\delta}_{t} $, the particles
are identically (but not independently) distributed and their
 marginal law is that of a Brownian motion at time $t$, so that
\begin{align*}
\PP^{\delta}( C^{(i)}_{t} > x \vert I^{\delta}_{t}) =  \Phi ( x / \sqrt{t})\color{black}\;,
\end{align*}
{with $\Phi(z) = \mathbf{P}(\mN \geqslant x)$ for a standard Gaussian $\mN$.}
Hence, conditional on the jump times, we obtain that for any $t \geq 0$ and $\eps>0$\color{black},
\[ 
\begin{aligned}\label{manytoonequenched}
\P^{\delta}(\mathbf S_t > x | I^{\delta}_{t} ) &= \P^{\delta}\big(\exists i \in [I^{\delta}_t] \colon
C_t^{(i)} > x \big| I^{\delta}_{t} \big) \\
& \leq \E^{\delta} \Big[ \sum_{i \in
[I^{\delta}_t]} \mathds
1_{\{C_t^{(i)}>x\}} \Big| I^{\delta}_{t} \Big] \mathds 1_{ \{I^{\delta}_{t}
 \leqslant \color{black} \mathrm e^{ (\mf{c}_{\delta} + \ve )
t}\}} + \mathds 1_{\{ I^{\delta}_{t}  > \color{black} \mathrm e^{(\mf{c}_{\delta} + \ve) t} \}} \\
&  \leqslant  I^{\delta}_t \,  \Phi ( x / \sqrt{t}) \mathds 1_{\{ I^{\delta}_{t}
\leqslant\mathrm e^{( \mf{c}_{\delta} + \ve) t } \}} +
\mathds 1_{\{ I^{\delta}_{t} >\mathrm e^{( \mf{c}_{\delta} + \ve) t} \}}\;.
\end{aligned} \]
Thus, the Gaussian tail bound
\[ \int_{t}^\infty\color{black} \frac{1}{\sqrt{2\pi}} \e^{-y^2/2} \d y \leqslant \frac{1}{\sqrt{2\pi}t} \e^{-\frac{t^2}{2}}, \qquad t \geq 0 \]
implies that for $\eps,\lambda>0$
\[ \P^{\delta}(\mathbf S_t > \lambda t | I^{\delta}_t) \mathds 1_{\{ I^{\delta}_{t}
\leqslant  \mathrm e^{( \mf{c}_{\delta} + \ve)
t} \}} \leq \frac{1}{\sqrt{2\pi t}\lambda} \e^{ ( \mf{c}_{\delta} + \ve) t-\frac{\lambda^2 t}{2}} \]
holds almost surely, whence for $\lambda>\sqrt{2( \mf{c}_{\delta} +\eps)}$
we obtain by Proposition~\ref{prop:globalsurvival} that almost surely
\begin{align*}
 \lim_{t\to\infty} \P^{\delta}(\mathbf S_t>\lambda t) \leqslant \lim_{t \to \infty}
\bigg\{ \frac{1}{\sqrt{2\pi t}\lambda} \e^{ ( \mf{c}_{\delta}+\eps) t-\frac{\lambda^2 t}{2}} +
\PP^{\delta} (I^{\delta}_{t}>\mathrm e^{(
\mf{c}_{\delta} + \ve) t}) \bigg\} =0 \;. \numberthis\label{Mtlimit}
\end{align*}
Since $ \ve >0 $ can be chosen arbitrarily small, the result follows.
\end{proof}

\subsection{Lower bound on the wave speed} 
To obtain the lower bound let us introduce the sequence of constants
\begin{equation}\label{eqn:c-delta-low}
\mf{r} + \int_{(\delta, 1]} \log{(1 + y)} \frac{1}{y} \mf{R}^{+}_{\delta}
(\ud y) =: \underline{\mf{c}}_{\delta} \;,
\end{equation}
which immediately satisfy $ \underline{\mf{c}}_{\delta} \leqslant
\frac{\mf{s}^{2}}{2}$.
Then the main result of this section is the next proposition.
\begin{proposition}\label{prop:rightmost-speed}
Fix $ \mf{R} \in \mM $ and, for any $ \delta \in (0, 1] $, let $
\mf{R}_{\delta} = (\mf{R}_{\delta}^{-}, \mf{R}_{\delta}^{+}) $ be as in
Definition~\ref{def:R-delta}. Let $ \mathbf{C}_{t} $ be an $ (\mf{r} \delta_{0}, \mf{R}_{\delta}^{+}) $--CBBM started in $
\mathbf{x} = 0 \in \RR^{1} $. Then for any $ 0< \lambda < \sqrt{2
\underline{\mf{c}}_{\delta}}  $ and $ x_{0} \in \RR $, we have $ \lim_{t \to \infty}
\PP^{\delta}( \mathbf{S}_{t} > \lambda t + x_{0}) =1 $, $\PP$--almost surely.
\end{proposition}

{We chose as approximating sequence the ordered pair of measures $\{(\mf{r}\delta_0, \mf{R}_\delta^+)\}_{\delta \in (0, 1]}$, so that the difference to the original measure is just given by $\mf{R}\vert_{(0, \delta]}$, which vanishes as $\delta \to 0$: In particular, this is why we included the mass at zero in our approximation.}
To prove this result, let us associate to the $ (\mf{r} \delta_{0},
\mf{R}^{+}_{\delta}) $--CBBM $ \mathbf{C}_{t} $ started in $ \mathbf{x} = 0 \in \RR^{1} $ a measure-valued
process 
\begin{equation}\label{eqn:mvp} 
X_{t} = \sum_{i = 1}^{n(t)} \delta_{x_{i}(t)} \;,  
\end{equation}
where we assume that at time $ t \geqslant 0 $, $ \mathbf{C}_{t} = (x_{1}(t), \dots, x_{n(t)}(t)) $.
Then, in the spirit of \cite{Kyprianou, Englander} we will link the wave speed to the local survival of the $
X_{t} ( \cdot + \lambda t) $.

\begin{proof}
Consider the measure-valued process $ X $ of \eqref{eqn:mvp}.
Now, let $ I \subseteq \RR $ be any compact interval (that is a set of the form $ I =
[a, b] $ for $ a < b $) and observe
that the following implication holds, as long as \( \lambda^{\prime} > \lambda \), for any $x_0 \in \RR $:
\begin{align*}
\left\{ \liminf_{t \to \infty} X_{t}(I + \lambda^{\prime} t) > 0 \right\}
\Longrightarrow \left\{ \liminf_{t \to \infty} \left( \mathbf{S}_{t} - \lambda t +
x_{0} \right) > 0 \right\} \;.
\end{align*}
In particular, our result follows if there exists a family $$\{ R_{\lambda, \zeta} \colon \quad \lambda \in (0, \sqrt{2
\underline{ \mf{c}}_{\delta}}) \;, \zeta \in (0,1)\}$$
of positive, $\mS_\delta$ adapted random variables such that for the intervals 
$$I_{\lambda, \zeta} = [- R_{\lambda, \zeta}, R_{\lambda, \zeta}]$$
the following is satisfied:
\begin{align*}
\PP^{\delta}\left( \liminf_{t \to \infty} X_{t}(I_{\lambda, \zeta} + \lambda t) >0
\right) \geqslant 1- \zeta, \quad \PP\text{--almost surely,} \quad \text{ for all } 0 < \lambda <
\sqrt{2 \underline{ \mf{c} }_{\delta}} \;.
\end{align*}
This is exactly the content of Lemma~\ref{lem:survival2} below. 
\end{proof}

\begin{lemma}\label{lem:survival2}
Let $ X $ be the measure-valued process of \eqref{eqn:mvp}. Then 
for any $ 0< \lambda < \sqrt{2 \underline{\mf{c}}_{\delta}}$ and $\zeta \in (0, 1)$ there exists an $\mS_\delta$--adapted positive random variable $R_{\lambda, \zeta}$
such that $ \PP $--almost surely $$\PP^{\delta} \big( \liminf_{t \to \infty}
X_{t}(I_{\lambda, \zeta} + \lambda t) > 0 \big) \geqslant 1 - \zeta \;,$$
with $I_{\lambda, \zeta} = [-R_{\lambda, \zeta}, R_{\lambda, \zeta}]$.
\end{lemma}

\begin{proof}

First of all we observe that instead of considering
$ X_{t}(A + \lambda t) $ for all measurable $A \subseteq \RR$ (that is shifting the sets we are measuring) we can and will consider
Brownian motions with a drift $ \lambda $ to the left in the definition of $X_{t}$ and measure $X_t(A)$ instead. Consider intervals of the form $ I(R) = [-R, R]$ for $ R > 0 $: below we will choose a
sufficiently large value $ R(\lambda, \zeta) $ and prove the desired result for $
I_{\lambda, \zeta} = I (R (\lambda, \zeta)) $.

\textit{Step 1: Continuity of the principal eigenvalue.} For any $ R > 0 $ let $ (P^{\lambda, R}_{t})_{t \geqslant 0} $
be the heat semigroup with drift $ \lambda $ and Dirichlet boundary conditions on $ \partial I(R) $, acting on
$ L^{2}(I(R)) $. Namely, for any $ \varphi \in L^{2}(I(R)) $, $
P_{t}^{\lambda, R} \varphi $ satisfies $ P_{0}^{\lambda, R} \varphi = \varphi$
and solves
\begin{align*}
  \partial_{t} P_{t}^{\lambda, R} \varphi (x) = \Big( \frac{1}{2} \Delta +
  \lambda \partial_{x} \Big) P^{\lambda,
  R}_{t} \varphi (x), \quad x \in (-R, R), \ \
  P_{t}^{\lambda, R} \varphi (\pm R) =0, \ \
  \forall t > 0.
\end{align*}
Let us denote with $ \mu(\lambda,R) = \sup \frac{1}{t} \log \sigma ( P^{\lambda,
R}_{t} ) $ the principal eigenvalue of $ \frac{1}{2} \Delta + \lambda
\partial_{x} $ on $ I(R) $ with Dirichlet boundary conditions
(here $ \sigma $ indicates the spectrum, and the definition of $ \mu $ does not depend on $ t >0 $). 
Then, we observe that by \cite[Section 4, Theorem 4.1]{PinskyBook}
\begin{equation}
  \lim_{R \to \infty} \mu(\lambda, R) =  - \frac{\lambda^{2}}{2} \;,
\end{equation}
the latter being the principal (generalised) eigenvalue of $ \frac{1}{2}
\Delta + \lambda \partial_{x} $ on $ \RR $.

In particular, for any $ \lambda < \lambda^\prime < \sqrt{2 \underline{\mf{c}}_{\delta}}$ we can find a $R_0(\lambda^\prime)$ such that  $\mu(\lambda, R) \geqslant - \frac{(\lambda^\prime)^2}{2}$ for all $R \geqslant R_0 (\lambda^\prime)$.

\textit{Step 2: Local survival.} We now consider $ \lambda^\prime $ and $ R_0(\lambda^\prime) $ as
above and define $ I_{0} = I_0 (\lambda, \lambda^\prime) = [- R_0 (\lambda^\prime), R_0 (\lambda^\prime)] $. Then we introduce a new process $ \overline{X}_{t} $ in which particles
evolve as in $ X_{t} $ but are killed (i.e.\ the disappear from the measure) upon reaching the
boundary of $ I_{0} $. By comparison we obtain that $
\overline{X}_{t} \leqslant X_{t} $ in the sense of positive measures. We will
then start by considering
\begin{equation*}
  \eta \stackrel{\rm{def}}{=} \PP^{\delta} \big( \liminf_{t \to \infty} \overline{X}_{t}(I_{0}) > 0 \big)\;,
\end{equation*}
and proving that $\eta$ is an $\mS_\delta$-adapted random variable satisfying
\begin{equation}\label{eqn:aim}
  \PP(\eta > 0 ) = 1\;.
\end{equation}
To prove this result let us fix $ \varphi $ the eigenfunction on $
L^{2}(I_{0}) $ associated to $ \mu = \mu(\lambda, R_0(\lambda^\prime)) $ (note that $
\varphi (x) > 0 $ for $ x \in (-R_0(\lambda^\prime), R_0(\lambda^\prime)) $ by the
Krein--Rutman theorem, and $ \varphi \in C^{\infty}_b( (- R_0(\lambda^\prime),
R_0(\lambda^\prime))) $ via classical regularity estimates), so that we can write the martingale problem for
$\overline{X}_{t}(\varphi) $ as follows. 
Next consider the jump times $ \{t_{j}\}_{j \in \NN} $ associated to $ \mS_{\delta} $. If we fix some $ j \in \NN $,
then we have that for $ t_{j} \leqslant t < t_{j+1} $ the process
\begin{align*}
  [t_{j}, t_{j+1}) \ni t \mapsto\mathrm e^{ - (\mf{r} +
\mu )t}\overline{X}_{t}(\varphi )  = M_{t}^{j}
\end{align*}
is a c\`adl\`ag martingale under $\PP^\delta$ on $ [t_{j}, t_{j+1}) $, with predictable quadratic variation
\begin{align*}
\langle M^{j} \rangle_{t} = \int_{t_{j}}^{t}\mathrm e^{- 2 (\mf{r} + \mu )r} \Big[ \mf{r} \, \overline{X}_{r} (
 \varphi^{2} ) + \overline{X}_{r}(
(\partial_{x} \varphi)^{2} ) \Big] \ud r\;,
\end{align*}
where the first term comes from independent reproduction and the second one
from the spatial motion of the particles. Next we consider the jumps at times
$ t_{j} $. We have
\begin{align*}
e^{-(\mf{r} + \mu)t_j} \left\{ \overline{X}_{t_{j}}(\varphi)- (1 + y_{j})
\overline{X}_{t_{j}-}(\varphi) \right\}= \Delta N_{j}\;,
\end{align*}
where $ \Delta N_{j} $ is a martingale increment. Since every particle alive a
time $ t_{j} $ reproduces with probability $ y_{j} $ independently of all other
particles, $ N_{j} $ has the
variance of a Bernoulli random variable with parameter $ y_{j} $:
\begin{align*}
\langle \Delta N_{j} \rangle = e^{-2(\mf{r} + \mu) t_j} y_{j}(1 - y_{j})
\overline{X}_{t_{j}-}(\varphi^{2})\;.
\end{align*}
Overall we can now conclude that the following is a c\`adl\`ag martingale  on $[0,\infty)$ under $\P_\delta$\color{black}:
\begin{align*}
L_{t} =\mathrm e^{- (\mf{r}+ \mu)t} \bigg( \prod_{j  \, \colon \, t_{j}
\leqslant t}
(1 + y_{j})^{-1}  \bigg) \overline{X}_{t}(\varphi)\;.
\end{align*}
In fact, for any $ 0 \leqslant s < t < \infty $ with $ j(t) \in \NN $ uniquely
defined by $ t \in [t_{j(t)}, t_{j(t)+1}) $, and assuming that $ j(s) < j(t) $
(otherwise the martingale property is inherited immediately from $
M^{j(t)} $) and
with the notation $ \Delta M_{j} = M^{j}_{t_{j+1}-} - M_{t_{j}}^{j} $, we find
that
\begin{align*}
L_{t} - L_{s} = & \bigg( \prod_{j  \leqslant j(t)}
(1 + y_{j})^{-1}  \bigg) \bigg[ (M^{j(t)}_{t} - M_{t_{j(t)}}^{j(t)}) + \Delta
N_{j} \bigg] \\
& + \sum_{\ell = j(s)+1}^{j(t)-1} \bigg( \prod_{j  \leqslant \ell}
(1 + y_{j})^{-1}  \bigg) \bigg[ \Delta M_{\ell} + \Delta N_{\ell} \bigg]\\
& + \bigg( \prod_{j  \leqslant j(s)}
(1 + y_{j})^{-1}  \bigg) (M^{j(s)}_{t_{j(s)+1}-} - M_{s}^{j(s)}) \;,
\end{align*}
which is a sum of martingale increments. Hence we have found a positive
martingale $ L_{t} $, which implies that there exists an almost sure limit $
\lim_{t \to \infty} L_{t} =  L_{\infty} \in [0, \infty) $. We want to prove
that $ \EE^{\delta} L_{\infty} = L_{0} > 0 $, which amounts to proving that the
martingale is uniformly integrable. Hence we will show that
\begin{equation}\label{eqn:aim-qv}
\sup_{t \geqslant 0} \EE^{\delta} L_{t}^{2} = L_{0} + \sup_{t \geqslant 0} \EE^{\delta} \langle
L \rangle_{t} < \infty\;.
\end{equation}
We are thus left with computing the expected quadratic variation.

Let us now follow the notation of Lemma~\ref{lem:asympt-skltn} and write $
e^{\mf{d}_{\delta}(t)} = \prod_{j \leqslant j(t)} (1 + y_{j}) $. Then for $
t \in [t_{j}, t_{j+1})$ we find that
\begin{align*}
\ud \langle L \rangle_{t} =\mathrm e^{- 2( \mf{r}  + \mu) t - 2
\mf{d}_{\delta}(t)} \bigg( \Big[ \overline{X}_{t}( \mf{r}\varphi^{2} +
( \partial_{x} \varphi)^{2} ) \Big] + y_{j}(1 - y_{j})
\overline{X}_{t_{j} -} ( \varphi^{2} ) \delta_{t_{j}} (t) \bigg) \ud t \;.
\end{align*}
The last ingredient  to bound the expected quadratic variation $ \EE^{\delta} \langle L_{t} \rangle $ is to bound
the expected value \color{black} of $ \overline{X}_{t} $. From the definition of $ \overline{X} $ we
find for $ s \in [0, t] $ and any $ \psi \in L^{2}(I_{0}) $
\begin{equation}\label{eqn:contrl}
\ud \EE^{\delta} \overline{X}_{s} ( P^{\lambda, R_0(\lambda^\prime)}_{t- s}\psi) = \mf{r} \EE^{\delta}
\overline{X}_{s} ( P^{\lambda, R_0(\lambda^\prime)}_{t- s}\psi) \ud s + \sum_{j \in \NN }
y_{j} \EE^{\delta}  \overline{X}_{s} ( P^{\lambda, R_0(\lambda^\prime)}_{t- s}\psi)
\delta_{t_{j}}(s) \ud s \;.
\end{equation}
Hence for $ \psi = \varphi $ we find
\begin{align*}
\EE^{\delta} \overline{X}_{t}(\varphi) =\mathrm e^{(\mf{r}+ \mu)t +
\sum_{j \leqslant j(t)} \log{ (1 +y_{j} )}} \varphi (0)\;.
\end{align*}
In particular, we can rewrite $ \EE^{\delta} \langle L \rangle_{t} $ as the sum of two
terms:
\begin{align*}
  \EE^{\delta} \langle L \rangle_{t} = & \int_{0}^{t}\mathrm e^{- 2( \mf{r}  + \mu) s - 2
  \mf{d}_{\delta}(s)} \bigg( \mf{r}\EE^{\delta} \overline{X}_{s}( \varphi^{2})  + y_{j}(1 - y_{j})
\EE^{\delta} \overline{X}_{t_{j} -} ( \varphi^{2} ) \delta_{t_{j}} (s) \bigg) \ud s \\
& + \int_{0}^{t}\mathrm e^{- 2( \mf{r}  + \mu) s - 2
  \mf{d}_{\delta}(s)} \EE^{\delta} \overline{X}_{s} ( ( \partial_{x} \varphi)^{2} )\ud s \;.
\end{align*}
The difference between the first and second line is that in the first line we
can estimate $ \overline{X}_{s}(\varphi^{2}) \leqslant \| \varphi \|_{\infty} \overline{X}_{s}(\varphi) $, which is not
possible for the term in the second line since $ \varphi (\pm R_0 (\lambda^\prime)) =0 $, which does not hold
for $ (\partial_{x} \varphi)^{2} $ (so here we will need some additional
arguments). To fix the key point of the proof let us start with the first
term. Using the previous computations we find
\begin{align*}
  \int_{0}^{t} &\mathrm e^{- 2( \mf{r}  + \mu) s - 2 \mf{d}_{\delta}(s)} \bigg(
  \mf{r}\EE^{\delta} \overline{X}_{s}( \varphi^{2})  + y_{j}(1 - y_{j}) \EE^{\delta}
\overline{X}_{t_{j} -} ( \varphi^{2} ) \delta_{t_{j}} (s) \bigg) \ud s \\
& \lesssim_{\mf{r}, \| \varphi \|_{\infty}} \int_{0}^{t} \exp \left\{ 
  \Big( - ( \mf{r}  + \mu)  - \mf{d}_{\delta} + o(1) \Big) s \right\} \ud s \;.
\end{align*}
Here the $ o(1) $ term is intended as $ s \to \infty $, and is a consequence of
Lemma~\ref{lem:asympt-skltn}. Now since by assumption $\lambda < \lambda^\prime < \sqrt{2\underline{\mf{c}}_{\delta}}$ we have
\begin{align*}
- \mf{r} - \mf{d}_{\delta} -\mu = - \underline{\mf{c}}_{\delta} - \mu \leqslant - \underline{\mf{c}}_{\delta} + \frac{(\lambda^\prime)^2}{2}
\overset{\rm{def}}{=} - \ve < 0 \;.
\end{align*}
In particular, the integral under consideration is
converging for $t \to \infty$. Now, if we pass to the term involving $ \EE^{\delta} \overline{X}_{s} ( ( \partial_{x} \varphi)^{2})
$, we find by \eqref{eqn:contrl}
\begin{align*}
\EE^{\delta} \overline{X}_{s} ( ( \partial_{x} \varphi)^{2} ) =
e^{\mf{r} s + \mf{d}_{\delta}(s)} P_{s}^{\lambda,
R_0(\lambda^\prime)}(\partial_{x} \varphi)^{2}(0)\;.
\end{align*}
Now at time $ s=1 $ we can control the semigroup $P_{1}^{\lambda,
R_0(\lambda^\prime)}(\partial_{x} \varphi)^{2} \leqslant C_{\lambda, R_0}
\varphi $, for some $ C_{\lambda, R_0}> 0 $: indeed both $P_{1}^{\lambda,
R_0(\lambda^\prime)}(\partial_{x} \varphi)^{2}$ and $\varphi$ are strictly positive in the interior of $I(R_0(\lambda^\prime))$, vanish at the boundary and are \emph{differentiable} at the boundary (differentiability follows for example from \cite[Theorem 1.1]{KRY}), so that the named constant must exist. Hence the term can be
controlled following the same arguments as above, observing that
\begin{align*}
\EE^{\delta} \overline{X}_{s} ( ( \partial_{x} \varphi)^{2} ) \leqslant
e^{\mf{r} s + \mf{d}_{0}s + \mu(s-1)} C_{\lambda, R_0}
\varphi(0)\;,
\end{align*}
which is of the same order as the bound used in the previous discussion.
Hence \eqref{eqn:aim-qv} is proven, and
since $ L_{0} > 0 $ we deduce  \eqref{eqn:aim} from the fact that $\PP$-almost surely
\begin{align*}
\PP^{\delta}( \liminf_{t \to \infty} X_{t}( I_{0} ) >0)
\geqslant \PP^{\delta} (\liminf_{t \to \infty} L_{t} > 0) > 0\;.
\end{align*}
Here we observe that the inclusion 
\begin{align*}
\left\{\liminf_{t \to \infty} L_{t} > 0 \right\} \subseteq \left\{ \liminf_{t
\to \infty} X_{t}( I_{0} ) >0\right\}
\end{align*}
holds because in the definition of $ L_{t} $ we find, by our assumptions on $
\mu $, that
\begin{align*}
e^{- (\mf{r}+ \mu)t} \bigg( \prod_{j  \, \colon \, t_{j} \leqslant t}
(1 + y_{j})^{-1}  \bigg)  \leqslant\mathrm e^{- \ve t + o(1)}\;,
\end{align*}
as $ t \to \infty $.

\textit{Step 3: Almost sure survival.} Now we want to use  \eqref{eqn:aim} to prove that if we choose a suitable larger random interval $I_{\lambda, \zeta}$, depending on $\lambda$ and $\zeta \in (0,1)$, then
$$\PP^{\delta} \big( \liminf_{t \to \infty} X_{t}(I_{\lambda, \zeta}) > 0 \big)\geqslant 1- \zeta\;, \qquad \PP\text{--almost surely}\;.$$
For this purpose let us write, for any $n\in \NN$ the interval
$$I^n_0=[-n R_0(\lambda), nR_0(\lambda)]\;,$$
and consider $\overline{X}^n_t$ the process in which particles
evolve as in $ X_{t} $ but are killed upon reaching the
boundary of $ I_{0}^n $. Next, note that we have
\begin{equation*}
   \PP^{\delta} \big( \liminf_{t \to \infty} \overline{X}^n_{t}(I_{0}^n) > 0 \big)= \PP^\delta (\tau^n = \infty)\;,
\end{equation*}
where $\tau^n$ is the extinction time $\tau^n = \inf \{ t \geq 0 \colon \; \overline{X}^n_t(I_0^n) = 0\}$ and the equality holds because for $t < \tau^n$ we have $\overline{X}_t^n (I_0^n) \geqslant 1$. Now let us prove that
\begin{equation}\label{e:aim-2}
    \PP^\delta (\tau^n < \infty) \leqslant \PP^\delta (\tau^1 < \infty)^n = (1 - \eta)^n\;,
\end{equation}
with $\eta$ as in \eqref{eqn:aim}.
Indeed, let us consider $ \overline{X}^{n,1}_t \leqslant \overline{X}^n_t$ the process in which particles
evolve as in $ \overline{X}^n_{t} $ but are killed upon reaching the
boundary of $ I_{0} $, coupled so that particles in $\overline{X}^{n,1}_t$ are exactly particles of $\overline{X}^n_t$ that have never left the interval $I_0$. By construction
\begin{align*}
    \tau^{n,1} < \tau^n \;,
\end{align*}
where $\tau^{n,1}$ is the extinction time of $\overline{X}^{n,1}$. This means that on the event $\tau^n < \infty$ we have at time $\tau^{n,1}$ at least one particle of $\overline{X}^n_{\tau^{n,1}}$ either in $-R_0(\lambda^\prime)$ or in $R_0(\lambda^\prime)$. Say the latter is the case and suppose that $n\geqslant 2$, then we can consider the process $\overline{X}^{n,2}_t \leqslant \overline{X}^n_t$ for $t \geqslant \tau^{n,1}$, started with exactly that particle in $R_0(\lambda^\prime)$ and in which particles are killed upon reaching the boundary of $[0, 2R_0(\lambda^\prime)]$. Observe that $\overline{X}^{n,2}_{\tau^{n,1}+t} (\cdot + R_0(\lambda^\prime))$ has the same law as $\overline{X}^{n,1}_t$ 
as in the previous step.  If we let $
\tau^{n, 2}$ be the extinction time of $\overline{X}^{n,2}$, then we obtain
\begin{align*}
    \PP^\delta ( \tau^n < \infty) & \leqslant \PP^\delta (\tau^{n,1} < \infty) \cdot \EE^\delta[ \PP^\delta (\tau^{n,2} < \infty \vert \mF_ {\tau^{n,1}})] \\
    & = \left[ \PP^\delta (\tau^1 < \infty) \right]^2= (1- \eta)^2\;,
\end{align*}
with $\mF_t$ the filtration generated by $\overline{X}^n$ and $\eta$ the random variable from \eqref{eqn:aim} (the last line follows from the strong Markov property). We can iterate this procedure at least $n$ times, so that \eqref{e:aim-2} is proven. If we choose $n = n(\eta, \zeta)$ (hence $n$ is random) so that
\begin{align*}
    (1- \eta)^n  \leqslant \zeta \;,
\end{align*}
then the claimed result follows.

\end{proof}

\subsection{Quenched growth rate}\label{sec:invexp}

 Our goal in this section is to verify Proposition~\ref{prop:globalsurvival}.
Recall that, as in Definition~\ref{def:R-delta}, $I_t^{\delta}= \ell(\mathbf C_t)$ for $t\geq 0$, where $(\mathbf C_t)_{t\geq
0}$ is an $ (\mf{R}_\delta^-, \mf{R}_\delta^+ ) $--CBBM, and recall that we use
the notation
\[ \mf{c}_{\delta} = \mf{R}^{-}_{\delta}([0, \delta]) + \int_{(\delta, 1]} \log{(1
+ y)} \frac{1}{y} \mf{R}^{+}_{\delta}( \ud y) \numberthis\label{eqn:cdef} \;. \] 
Now, Proposition~\ref{prop:globalsurvival} is equivalent to the next lemma.
\begin{lemma}\label{lem:speed-support}
In the setting of Proposition~\ref{prop:globalsurvival}, for any $ \ve >0 $ we
have that $ \PP $--almost surely
\begin{align*}
\PP^{\delta}( \limsup_{t \to \infty} I_{t}^{\delta} e^{- (\mf{c}_{\delta} +
\ve)t} > 0 )& = 0 \;.
\numberthis\label{speed+eps}
\end{align*}
\end{lemma}
The proof will follow two different arguments for small and large
jumps. For small jumps we use a martingale approach: this leads to the term $
\mf{R}^{-}_{\delta}([0, \delta]) $ in our wave speed upper bound. In
particular, the martingale argument is not exact and delivers only a rough
upper bound (but as $ \delta \downarrow 0 $ this error will be
negligible). Instead for large jumps our argument is exact and builds on a time
change argument, which is possible since jump times are now discrete.

\begin{proof}
For brevity, let us write $ I_{t} $ for $ I_{t}^{\delta} $ and denote with $
j(t) = \max \{ j  \ \colon \ t_{j} \leqslant t \} $. We have
\begin{align*}
I_{t} = \frac{I_{t}}{I_{t_{j(t)}}} \cdot \bigg( \prod_{j \leqslant j(t) } \frac{I_{t_{j}-}}{I_{t_{j-1}}}  \bigg) \cdot \bigg( \prod_{j \leqslant j(t)}
\frac{I_{t_{j}}}{I_{t_{j}-}} \bigg) \cdot I_{0} \;,
\end{align*}
with the convention that $t_0=0$. \color{black}
Hence our result will follow if we prove the following three inequalities
\begin{align*}
\limsup_{t \to \infty} \frac{1}{t} \sum_{j = 1}^{j(t)} \log{\frac{I_{t_{j}-}}{I_{t_{j-1}}}}
& \leqslant \mf{R}_{\delta}([0, \delta]) \;, \\
\limsup_{t \to \infty} \frac{1}{t} \sum_{j = 1}^{j(t)} \log{\frac{I_{t_{j}}}{I_{t_{j}-}}}
& \leqslant \int_{( \delta, 1]} \log{(1 + y)} \frac{1}{y}
\mf{R}^{+}_{\delta} (\ud y) \;, \\
\limsup_{t \to \infty} \frac{1}{t} \log{\frac{I_{t}}{I_{t_{j(t)}}}} & = 0 \;,
\end{align*}

with the convention that $t_0=0$\color{black}.
We observe that the last equality 
follows analogously to the first inequality, so we restrict to proving the first two points.

\textit{Step 1: Martingale term.} Let us start by proving the first bound. We can define the following discrete
time process:
\begin{align*}
M_{n} =\mathrm e^{- t_{n} \mf{R}_{\delta}([0, \delta])}
\prod_{j=1}^{n}\frac{I_{t_{j}-}}{I_{t_{j-1}}}\;,
\end{align*}
and we observe that $ (M_{n})_{n \in \NN_0\color{black}} $ is a discrete-time \color{black} $\mathbb P^\delta$-\color{black}martingale with
respect to the filtration $ (\mF_{t_{n}} )_{n \in \NN_0\color{black}}, $ where $
\mF_{t} $ is the filtration generated by $ (\mathbf{C}_{s})_{s \leqslant t} $.
To see that the martingale property holds, we observe that $
\mathbf{C}_{t} $ has the law of an $ (\mf{R}_\delta^-, 0) $--CBBM
on every time interval $ [t_{j}, t_{j+1}) $. In particular, by
Proposition~\ref{prop-invexp}, we see that
\begin{align*}
\EE^{\delta\color{black}} \big[ M_{n} \vert \mF_{t_{n-1}\color{black}} \big] =\mathrm e^{\mf{R}_{\delta}([0, \delta])(t_{n} -
t_{n-1})}\mathrm e^{-t_{n} \mf{R}_{\delta}([0, \delta])} \prod_{j =1}^{n-1}
\frac{I_{t_{j}-}}{I_{t_{j-1}}}= M_{n-1}.
\end{align*}
Since $ M_{n} $ is a positive martingale, it follows from the martingale
convergence theorem and Fatou's lemma that
\begin{align*}
\EE^{\delta\color{black}} \big[ \lim_{n \to \infty} M_{n} \big] \leqslant \liminf_{n \to \infty} \EE^{\delta\color{black}}
M_{n} \leqslant \EE^{\delta\color{black}} M_{0} =1 \;.
\end{align*}
We conclude that almost surely, for any $\ve > 0$
\begin{align*}
\limsup_{n \to \infty} \frac{1}{t_{n}} \sum_{j=1}^{n} \log{
\frac{I_{t_{j}-}}{I_{t_{j-1}}} } \leqslant \mf{R}_{\delta}([0, \delta]) +\ve\;,
\end{align*}
which proves the first bound (note that $ \frac{t_{j(t)}}{t} \to 1 $ as $ t \to
\infty $).

\textit{Step 2: Large jumps.} Here we use a different argument, based on large
deviation principles. In a nutshell, we will prove that as long as the solution
is growing exponentially fast at the correct order, such exponential growth
becomes ever more likely (and precise) in future. We observe that the process
$ I_{t} $ does not depend on the spatial dynamics of the particles. In
particular, for every $ j \geqslant 1 $, the increment
$ \frac{I_{t_{j}}}{I_{t_{j}-}} $
depends only on the number $ R_{j} $ of particles that participate in the $j$-th
reproduction event:
\[  \frac{I_{t_{j}}}{I_{t_{j}-}} = 1 + \frac{R_{j}}{I_{t_{j}-}}\;.\] 
Recall also that every particle reproduces independently of any other particle
at time $ t_{j} $, with probability $ y_{j} \in (\delta, 1] $. Our aim is then
to prove that as $ I_{t_{j}-} $ increases, the approximation
\[ \frac{R_{j}}{I_{t_{j}-}} \simeq y_{j} \] 
becomes ever more likely and precise.
Following this description, let us consider, for some \( \ve \in (0, 1] \)
and $ M \in \NN $
\[ \mathcal{G}_{0, M} := \Big\{  \Big\vert \frac{R_{j}}{I_{t_{j}-}} - y_{j}
\Big\vert \leqslant \ve, \ \forall j  \in \{ 1, \dots,
M \} \Big\} \;, \]
where the letter $ \mG  $ stands for being a ``good'' set. We can then find a
$ c (\ve)>0 $ such that for $I_0 \geqslant 1$
\begin{align*}
\PP^{\delta} ( \mathcal{G}_{0,M} ) & \geqslant  \PP^{\delta}(\mG_{0, M} \vert \mG_{0, M-1})
\PP^{\delta}(\mG_{0, M-1} )\\
& \geqslant \left(1 - \exp\left\{ - c (\ve) I_{0} \prod_{j =1}^{M-1}(1
+y_{j} - \ve) \right\} \right) \PP^{\delta} (\mG_{0, M-1}) \;,
\end{align*}
where we used that on the set $ \mG_{0, M-1} $ we have $ I_{t_{M}-} \geqslant
I_{0} \prod_{j =1 }^{M-1}(1 +y_{j}- \ve) $ (note that all other reproduction
events, not due to large jumps, only  increase the value of $
I_{t} $), together with the following large deviations bound \eqref{e:aim-3} for an i.i.d.\ sequence $\{X_i\}_{i \in \NN}$ of Bernoulli random variables of parameter $p$:
\begin{equation}\label{e:aim-3}
 \forall n \geqslant 1\;, \ve \in (0, 1]\;, p \in [0,1]\;, \quad \exists c(\ve) > 0 \quad \text{s.t.} \quad \PP \Big( \Big\vert \frac{1}{n} \sum_{i=1}^{n} X_{i} - p \Big\vert \geqslant \ve
\Big) \leqslant\mathrm e^{- c( \ve) n }\;.
\end{equation}
Indeed, for $p \in [0,1]$ and $\eps \in (0,1]$ we have that for all $n \in \NN \setminus \{0\}$
\begin{equation}\label{Cramér}
\begin{aligned}
  \P \Big( \Big| \frac{1}{n} \sum_{i=1}^n X_i - p \Big| \geq \eps \Big)
 &= \P \Big(  \frac{1}{n} \sum_{i=1}^n X_i  \leq p-\eps \Big)+ \P \Big(  \frac{1}{n} \sum_{i=1}^n X_i  \geq p+\eps \Big) 
    \\ & \leqslant \e^{-n \bar c(p-\eps,p)}+\e^{- n\bar c(p+\eps,p)},
\end{aligned}
\end{equation}
where the rate function is given by the relative entropy
\[ \bar c(a,b) = \begin{cases}
a \log \frac{a}{b} - (1-a) \log \frac{1-a}{1-b}, &\text{ if } a \in [0,1], \\
\infty, &\text{ otherwise}, \end{cases}  \]
for $b\in[0,1]$, using the convention that $0\log 0 = 0\log \frac{0}{0}=0$. Here the second line of~\eqref{Cramér} follows from Markov's inequality.
Thus, in order to show \eqref{e:aim-3}
it suffices to verify that 
\begin{equation*}
c(\ve) \overset{\mathrm{def}}{=}\inf_{p \in [0,1]} \min \Big\{ \bar c(p-\eps,p), \bar c(p+\eps,p) \Big\}>0\;, \end{equation*}
which follows from the observation that $\bar c(a,b)=0$ if and only if $a=b \in [0,1]$, since $a \mapsto \bar c(a,b)$ is convex for fixed $b \in [0,1]$, and since $(a,b) \mapsto \bar c(a,b)$ is continuous on $(0,1)^2$.
Iterating the previous bound down to $ M =1 $ we obtain, always assuming $
I_{0} \geqslant 1$
\begin{align*}
\PP^{\delta} (\mG_{0, M}) & \geqslant \prod_{n = 1}^{M} \left(1 - \exp\left\{ - c (\ve)
I_{0} \prod_{j =1}^{n-1}(1 +y_{j} - \ve) \right\} \right) \\
& \geqslant 1 - \sum_{n=1}^{M} \exp\left\{ - c (\ve) I_{0} \prod_{j =1}^{n-1}(1
+y_{j} - \ve) \right\}  \;.
\end{align*}
Here we made use of the inequality $ \prod_{n=1}^{M}(1 - \alpha_{n}) \geqslant
1 - \sum_{n=1}^{M} \alpha_{n} $, which holds for any sequence $
\alpha_{n} \in (0, 1) $ by iterating the following bound:
\begin{align*}
\prod_{n = 1}^{M} (1 - \alpha_{n}) = \prod_{n=2}^{M} (1 - \alpha_{n}) -
\alpha_{1} \prod_{n =2}^{M} (1 - \alpha_{n}) \geqslant
\prod_{n=2}^{M}(1 - \alpha_{n}) - \alpha_{1} \;.
\end{align*}
Finally, by monotonicity, with $ \mG_{0, \infty} = \bigcap_{M \in \NN}
\mG_{0, M} $, we find that for $\ve \in (0, \delta/2)$
\begin{align*}
I_{0} \geqslant n_{0} \quad \Longrightarrow \quad \PP^{\delta}(\mG_{0 , \infty})
\geqslant  1 - \sum_{n = 1}^{\infty} \exp\left\{ - c (\ve) n_0 \prod_{j =1}^{n-1}(1
+y_{j} - \ve) \right\}  \overset{\rm{def}}{=} \mf{p} (n_0, \ve) \in \RR\;.
\end{align*}
The fact that the sum is
converging follows for example from the condition $\ve \in (0, \delta/2)$, so that it can be bounded from below by $1 - \sum_n \exp \{- c(\ve) n_0 (1 + \delta/2)^n\}> -\infty$. In particular, since for any $n_0 \in \NN$ and $I_0 \geqslant 1$ we have
\[\PP ( \exists j \in \NN \text{ such that } I_{t_j} \geqslant n_0) = 1\;,\]
we obtain by the strong Markov property that for any $\ve \in (0, \delta/2)$ and $\mf{p}(n_0, \ve)$ as above
\begin{align*}
    \PP \left( \left\vert \frac{R_j}{I_{t_j-}} - y_j \right\vert \leqslant \ve\;, \text{ for all but finitely many } j \in \NN \right) \geqslant \mf{p}(n_0, \ve) \;, \qquad \forall n_0 \in \NN\;.
\end{align*}
Now, since $\lim_{n_0 \to \infty} \mf{p} (n_0, \ve) = 1$ for all $\ve \in (0, \delta/2)$, we conclude
\begin{align*}
    \PP \left( \left\vert \frac{R_j}{I_{t_j-}} - y_j \right\vert \leqslant \ve\;, \text{ for all but finitely many } j \in \NN \right) = 1\;.
\end{align*}
We therefore deduce that on a set of full probability
\begin{align*}
\limsup_{ t \color{black} \to \infty} \frac{1}{t} \sum_{j =1}^{j(t)} \log
\frac{I_{t_{j}}}{I_{t_{j}-}} & \leqslant \int_{(\delta, 1]} \log{(1 + y + \ve)}
\frac{1}{y} \mf{R}_{\delta}^{+}( \ud y)\\
& \leqslant \int_{(\delta, 1]} \log{(1 + y)}
\frac{1}{y} \mf{R}_{\delta}^{+}( \ud y) + \ve \int_{(\delta, 1]}
\frac{1}{1 + y} \frac{1}{y} \mf{R}_{\delta}^{+}(\ud y) \;.
\end{align*}
Since $\eps>0$ was arbitrary, this \color{black} concludes the proof of the second inequality, and thereby of the lemma.
\end{proof}


\begin{thebibliography}{10}

\bibitem{BAH}
H. Bahouri, J.-Y. Chemin, and R. Danchin.
\newblock {\em Fourier analysis and nonlinear partial differential equations},
  volume 343 of {\em Grundlehren der mathematischen Wissenschaften [Fundamental
  Principles of Mathematical Sciences]}.
\newblock Springer, Heidelberg, 2011.

\bibitem{BCM}
V.~Bansaye, M.-E. Caballero, and S.~M\'{e}l\'{e}ard.
\newblock Scaling limits of population and evolution processes in random
  environment.
\newblock {\em Electron. J. Probab.}, 24:paper no. 19, 38, 2019.

\bibitem{BEV}
N.~H. Barton, A.~M. Etheridge, and A.~V\'{e}ber.
\newblock A new model for evolution in a spatial continuum.
\newblock {\em Electron. J. Probab.}, 15:no. 7, 162--216, 2010.

\bibitem{Berestycki}
J.~Berestycki.
\newblock {\em Topics on Branching Brownian motion}.
\newblock available at:
  \url{http://www.stats.ox.ac.uk/~berestyc/Articles/EBP18\_v2.pdf}, 2014.

\bibitem{BHN20}
J.~Blath, M.~Hammer, and F.~Nie.
\newblock The stochastic {F}isher-{K}{P}{P} equation with seed bank and on/off
  branching-coalescing {B}rownian motion.
\newblock {\em arXiv:2005.01650}, 2020.

\bibitem{BJN21}
J.~Blath, D.~Jacobi, and F.~Nie.
\newblock How the interplay of dormancy and selection affects the wave of
  advance of an advantageous gene.
\newblock {\em arXiv:2106.08655}, 2021.

\bibitem{CV19}
F.~Cordero and G.~Véchambre.
\newblock Moran model and {W}right--{F}isher diffusion with selection and
  mutation in a one-sided random environment.
\newblock {\em arXiv:1911.12089}, 2019.

\bibitem{CHL}
A.~Cortines, L.~Hartung, and O.~Louidor.
\newblock The structure of extreme level sets in branching {B}rownian motion.
\newblock {\em Ann. Probab.}, 47(4):2257--2302, 2019.

\bibitem{lizard}
C.~M.~Donihue, A.~Herrel, A.~C.~Fabre, A.~Kamath, A.~Geneva, T.~Schoener,
J.~Kolbe, and J.~Losos. 
\newblock Hurricane-induced selection on the morphology of an
island lizard. 
\newblock {\em Nature}, 560(7716):88–91, 2018.


\bibitem{Englander}
J.~Engl\"{a}nder and A.~E. Kyprianou.
\newblock Local extinction versus local exponential growth for spatial
  branching processes.
\newblock {\em Ann. Probab.}, 32(1A):78--99, 2004.

\bibitem{GKT}
A.~Gonz\'{a}lez~Casanova, N.~Kurt, and A.~T\'{o}bi\'{a}s.
\newblock Particle systems with coordination.
\newblock {\em ALEA Lat. Am. J. Probab. Math. Stat.}, 18(2):1817--1844, 2021.

\bibitem{GCSW}
A.~González~Casanova, D.~Span\`o, and M.~Wilke-Berenguer.
\newblock The effective strength of selection in random environment.
\newblock {\em arXiv:1903.12121}, 2019.

\bibitem{Griffeath}
D.~Griffeath.
\newblock The binary contact path process.
\newblock {\em Ann. Probab.}, 11(3):692--705, 1983.

\bibitem{Ikeda1}
N.~Ikeda, M.~Nagasawa, and S.~Watanabe.
\newblock Branching {M}arkov processes. {I}.
\newblock {\em J. Math. Kyoto Univ.}, 8:233--278, 1969.

\bibitem{Ikeda2}
N.~Ikeda, M.~Nagasawa, and S.~Watanabe.
\newblock Branching {M}arkov processes. {II}.
\newblock {\em J. Math. Kyoto Univ.}, 8:365--410, 1969.

\bibitem{Ikeda3}
N.~Ikeda, M.~Nagasawa, and S.~Watanabe.
\newblock Branching {M}arkov processes. {III}.
\newblock {\em J. Math. Kyoto Univ.}, 9:95--160, 1969.

\bibitem{K93}
J.~F.~C. Kingman.
\newblock {\em Poisson processes}, volume~3 of {\em Oxford Studies in
  Probability}.
\newblock The Clarendon Press, Oxford University Press, New York, 1993.
\newblock \href{http://www.ams.org/mathscinet-getitem?mr=MR1207584}{MR1207584}.

\bibitem{KR}
A.~Klimek and T.~Rosati.
\newblock The spatial {$\Lambda$}-{F}leming-{V}iot process in a random
  environment.
\newblock To appear in Ann. Appl. Probab., http://arxiv.org/abs/2004.05931v2.

\bibitem{KN97}
S.~M.~Krone and C.~Neuhauser.
\newblock Ancestral processes with selection.
\newblock {\em Theor. Popul. Biol.}, 51(3):210--237, 1997.

\bibitem{KRY}
N.~V. Krylov.
\newblock Boundedly inhomogeneous elliptic and parabolic equations in a domain.
\newblock {\em Izv. Akad. Nauk SSSR Ser. Mat.}, 47(1):75--108, 1983.

\bibitem{Kueh}
C.~Kuehn.
\newblock Travelling waves in monostable and bistable stochastic partial
  differential equations.
\newblock {\em Jahresber. Dtsch. Math.-Ver.}, 122(2):73--107, 2020.

\bibitem{Kyprianou}
A.~E. Kyprianou.
\newblock Asymptotic radial speed of the support of supercritical branching
  {B}rownian motion and super-{B}rownian motion in {${\mathbb{R}}^d$}.
\newblock {\em Markov Process. Related Fields}, 11(1):145--156, 2005.

\bibitem{Lalley-Sellke}
S.~P. Lalley and T.~Sellke.
\newblock A conditional limit theorem for the frontier of a branching
  {B}rownian motion.
\newblock {\em Ann. Probab.}, 15(3):1052--1061, 1987.

\bibitem{MM18}
B.~Mallein and P.~Miłoś,
\newblock Maximal displacement of a supercritical branching random
walk in a time-inhomogeneous random environment.
\newblock \emph{Stochastic Process.\ Their Appl.}, 129(9):3239--3260 (2019).

\bibitem{McKean}
H.~P. McKean.
\newblock Application of {B}rownian motion to the equation of
  {K}olmogorov-{P}etrovskii-{P}iskunov.
\newblock {\em Comm. Pure Appl. Math.}, 28(3):323--331, 1975.

\bibitem{antibiotic}
R.~Peña-Miller, A.~Fuentes-Hernandez, C.~Reding, I.~Gudelj, and R.~Beardmore. \newblock Testing the optimality properties of a dual antibiotic treatment in a two-locus, two-allele model. 
\newblock {\em Journal of The
Royal Society Interface}, 11(96):20131035, 2014.

\bibitem{MuellerMytnikQuastel}
C.~Mueller, L.~Mytnik, and J.~Quastel.
\newblock Effect of noise on front propagation in reaction-diffusion equations
  of {KPP} type.
\newblock {\em Invent. Math.}, 184(2):405--453, 2011.

\bibitem{PR}
N.~Perkowski and T.~Rosati.
\newblock A rough super-{B}rownian motion.
\newblock {\em Ann. Probab.}, 49(2):908--943, 2021.

\bibitem{PinskyBook}
R.~G. Pinsky.
\newblock {\em Positive harmonic functions and diffusion}, volume~45 of {\em
  Cambridge Studies in Advanced Mathematics}.
\newblock Cambridge University Press, Cambridge, 1995.

\bibitem{pitman}
J.~Pitman.
\newblock Coalescents with multiple collisions.
\newblock {\em Ann. Probab.}, 27(4):1870--1902, 1999.

\bibitem{Sagitov}
S.~Sagitov.
\newblock The general coalescent with asynchronous mergers of ancestral lines.
\newblock {\em J. Appl. Probab.}, 36(4):1116--1125, 1999.

\end{thebibliography}
\bibliographystyle{plain}

\begin{acks}
TR gratefully acknowledges financial support through the Royal Society grant 
RP\textbackslash R1\textbackslash 191065. We thank J.~Blath,  M.~Hammer, \color{black} N.~Hansen, F.~Hermann, A.~González Casanova, N.~Kurt,  B.~Mallein, \color{black} F.~Nie and M.~Wilke Berenguer for interesting discussions and comments. We also thank two anonymous reviewers for insightful suggestions and corrections.
\end{acks}



\end{document}